\documentclass[11pt]{amsart}

\usepackage{amsfonts}
\usepackage{amsmath,amssymb,amsthm}
\usepackage{enumerate,graphicx,xypic}
\usepackage[square,numbers]{natbib}

\newtheorem{theorem}{Theorem}[section]

\newtheorem{proposition}[theorem]{Proposition}
\newtheorem{corollary}[theorem]{Corollary}
\newtheorem{lemma}[theorem]{Lemma}

\theoremstyle{definition}

\newtheorem{definition}[theorem]{Definition}
\newtheorem{example}[theorem]{Example}

\newtheorem{remark}[theorem]{Remark}

\begin{document}

\title{Reversibility Problem of Multidimensional Finite Cellular Automata}


\author{Chih-Hung Chang}
\address[Chih-Hung Chang]{Department of Applied Mathematics, National University of Kaohsiung, Kaohsiung 81148, Taiwan, ROC.}
\email{chchang@nuk.edu.tw}

\author{Jing-Yi Su}
\address[Jing-Yi Su]{Department of Applied Mathematics, National University of Kaohsiung, Kaohsiung 81148, Taiwan, ROC.}

\author{Hasan Ak{\i}n}
\address[Hasan Ak{\i}n]{Ceyhun Atuf Kansu Street, Cankaya-Ankara, Turkey}

\author{Ferhat \c{S}ah}
\address[Ferhat \c{S}ah]{Department of Mathematics, Ad{\i}yaman University, Ad{\i}yaman, Turkey}

\thanks{This work is partially supported by the Ministry of Science and Technology, ROC (Contract No MOST 105-2115-M-390 -001 -MY2).}
\date{April 3, 2017}

\baselineskip=1.5\baselineskip

\begin{abstract}
While the reversibility of multidimensional cellular automata is undecidable and there exists a criterion for determining if a multidimensional linear cellular automaton is reversible, there are only a few results about the reversibility problem of multidimensional linear cellular automata under boundary conditions. This work proposes a criterion for testing the reversibility of a multidimensional linear cellular automaton under null boundary condition and an algorithm for the computation of its reverse, if it exists. The investigation of the dynamical behavior of a multidimensional linear cellular automaton under null boundary condition is equivalent to elucidating the properties of the block Toeplitz matrix. The proposed criterion significantly reduces the computational cost whenever the number of cells or the dimension is large; the discussion can also apply to cellular automata under periodic boundary conditions with a minor modification.
\end{abstract}

\maketitle

\section{Introduction} \label{sec:introduction}

A frequently used technique for studying complex structure is dividing the system into smaller pieces that are elaborated accordingly; for most physical systems, the dynamical behavior usually depends on the interactions among their neighbors. Cellular automaton (CA), introduced by Ulam and von Neumann, is a particular class of discrete dynamical system consisting of a regular network of cells which change their states simultaneously according to the states of their neighbors under a local rule; this makes CA an appropriate approach to model systems with the above mentioned property. CAs has been found applications in simulating or modeling complex systems in diverse areas such as vehicular ad hoc networks modeling, pattern formation, cryptography, image processing and image coding \cite{ALI+-CNSNS2013, CMC+-ITIP2011, HT-ITNN2011, KHP-IATCBB2012, KBT+-JID2013, SUA+-AMM2015, VBT-IJSAC2011}.

\begin{figure}
\begin{center}
\includegraphics[scale=0.45]{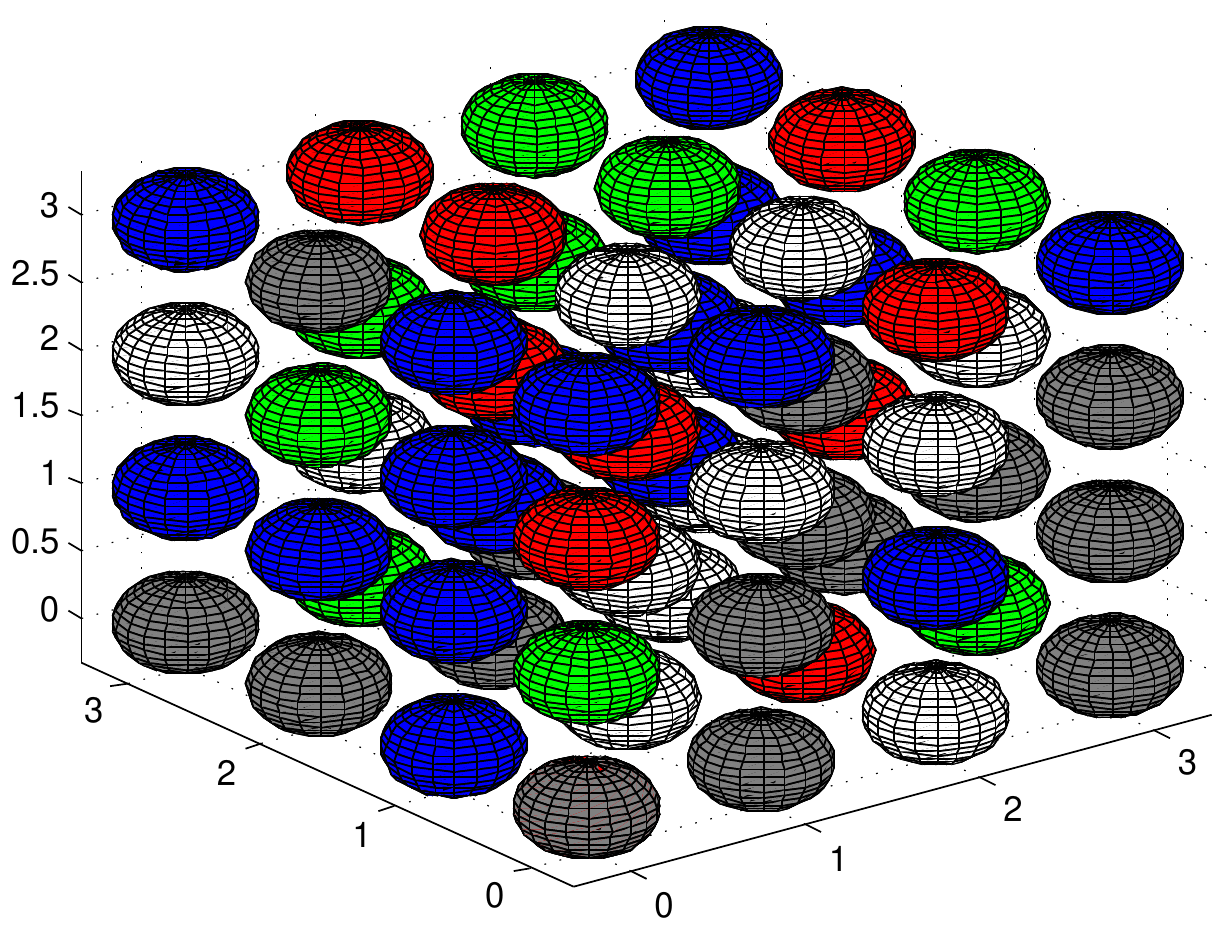} \quad
\includegraphics[scale=0.45]{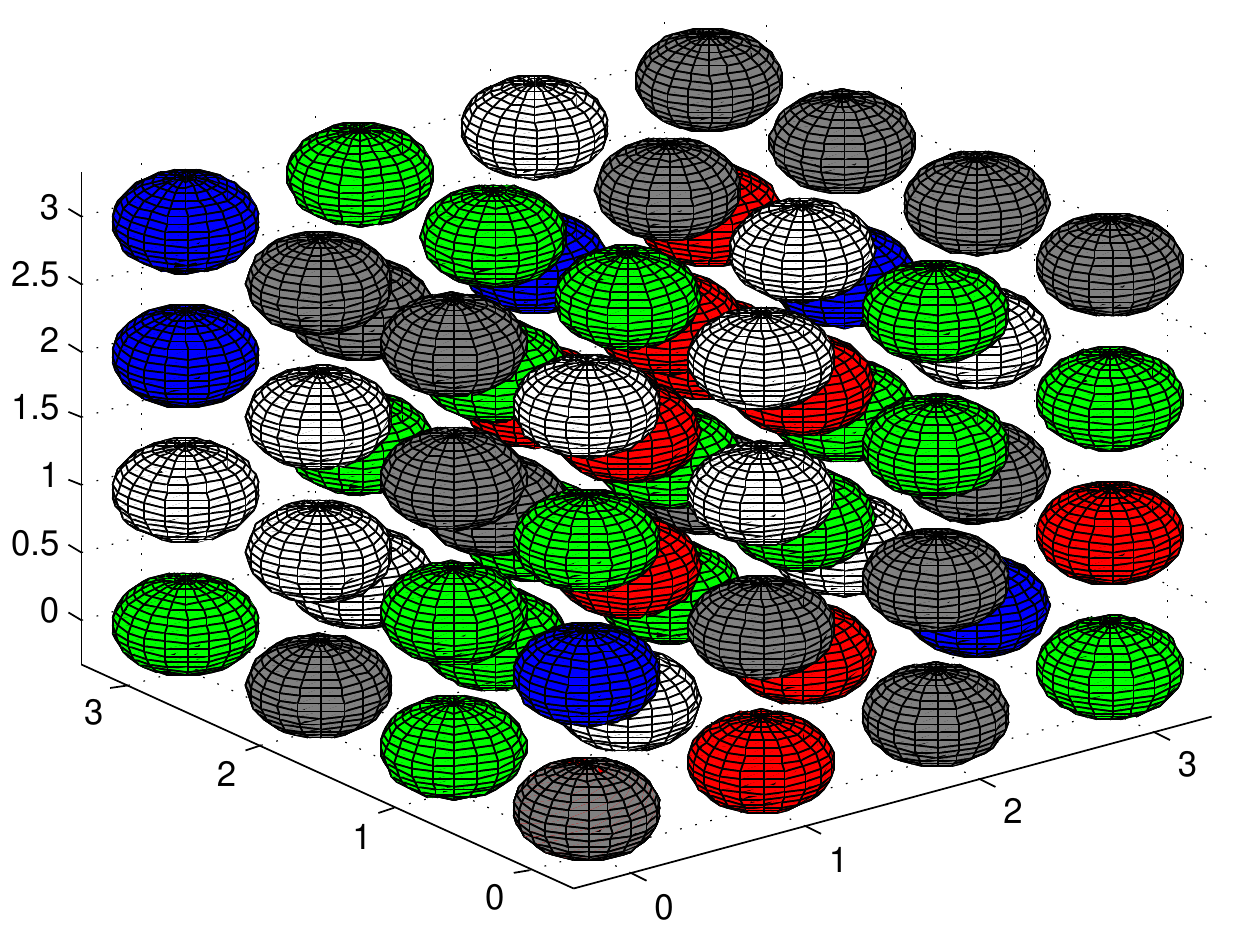} \\
(a) \qquad \qquad \qquad \qquad \qquad \qquad \qquad \qquad (b)
\end{center}
\caption{A three-dimensional linear cellular automaton $\Phi$ which is elucidated in Example \ref{eg:m=n=s=4-p=5-semiJordan}. The $64$ cells are located in a $4 \times 4 \times 4$ cube, and the states $0, 1, 2, 3$, and $4$ are represented by white, red, green, blue, and gray, respectively. The pattern (b) is seen as the $30$th evolution of (a).}
\label{fig:initial-30th-pattern-example-mns4-p5}
\end{figure}

One of the fundamental microscopic properties of nature is physical reversibility, which is a motivation for studying the reversibility problem of a dynamical system \cite{Morit-2012}; reversible computing systems are defined as each of their computational configurations has at most one previous configuration; this makes every computation process can be traced backward uniquely. In other words, reversible computing systems are deterministic in both directions of time. Figure \ref{fig:initial-30th-pattern-example-mns4-p5} illustrates the patterns of an initial pattern and its $30$th evolution of a three-dimensional linear cellular automaton (see Example \ref{eg:m=n=s=4-p=5-semiJordan} for more details). Landauer's principle asserts that an irreversible logical operation, such as erasure of an unnecessary information, inevitably causes heat generation \cite{Landauer-IJRD1961}; Bennett revealed that each irreversible Turing machine can be realized by a reversible one that simulates the former and leaves no garbage information on its tape when it halts \cite{Bennett-IJRD1973}; later on, Toffoli demonstrated that every irreversible $d$-dimensional cellular automaton can be simulated by some $(d+1)$-dimensional reversible cellular automaton \cite{Toffoli-JCSS1977}, and Morita and Harao revealed computation-universality of one-dimensional reversible cellular automata \cite{MH-IT1989}. Recently, it has been shown that a reversibly linear CA is either a Bernoulli automorphism or non-ergodic \cite{CC-IS2016}.

While the reversibility of one-dimensional CAs is elucidated \cite{AP-JCSS1972, Morit-2012, Nasu-TAMS2002}, Kari indicated that the reversibility of multidimensional CAs is generally undecidable \cite{Kari-PD1990, Kari-JCSS1994}. When restricted to linear CAs, however, the reversibility problem of multidimensional systems is demonstrated. More specifically, aside from the necessary and sufficient condition, an explicit formula of the inverse of a multidimensional (reversible) linear CA is given (cf.~\cite{ION-JCSS1983, MM-JCSS1998}). This paper studies the reversibility of linear CAs under boundary conditions and demonstrates an algorithm for computing the inverse, which helps for investigating some classical problems such as the Garden of Eden. For more information about the reversibility problem of CAs, the reader is referred to \cite{Kari-TCS2005, Morit-2012} and the references therein.

Recently, the reversibility problem of CA under boundary conditions has been widely studied since the number of cells is usually finite in practical applications. As the reversibility problem of one-dimensional cellular automata under periodic boundary conditions is generally answered, there are relatively few results about both one-dimensional and multidimensional cellular automata under null boundary conditions (see \cite{CAS-JSP2011, RS-AMC2011, ER-AMC2007a, KSA-TJM2016, NY-JPAG2004, YWX-IS2015} and the references therein); beyond that, investigations about the reversibility problem of cellular automata with memory and $\sigma$-automata are also seen in the literature \cite{TMA+-IS2012, Yamagishi-FFA2013}. Meanwhile, the reversibility problem of cellular automata on Cayley trees under periodic boundary conditions is studied in \cite{CS-2016}.

\begin{figure}
\begin{center}
\includegraphics[scale=0.45]{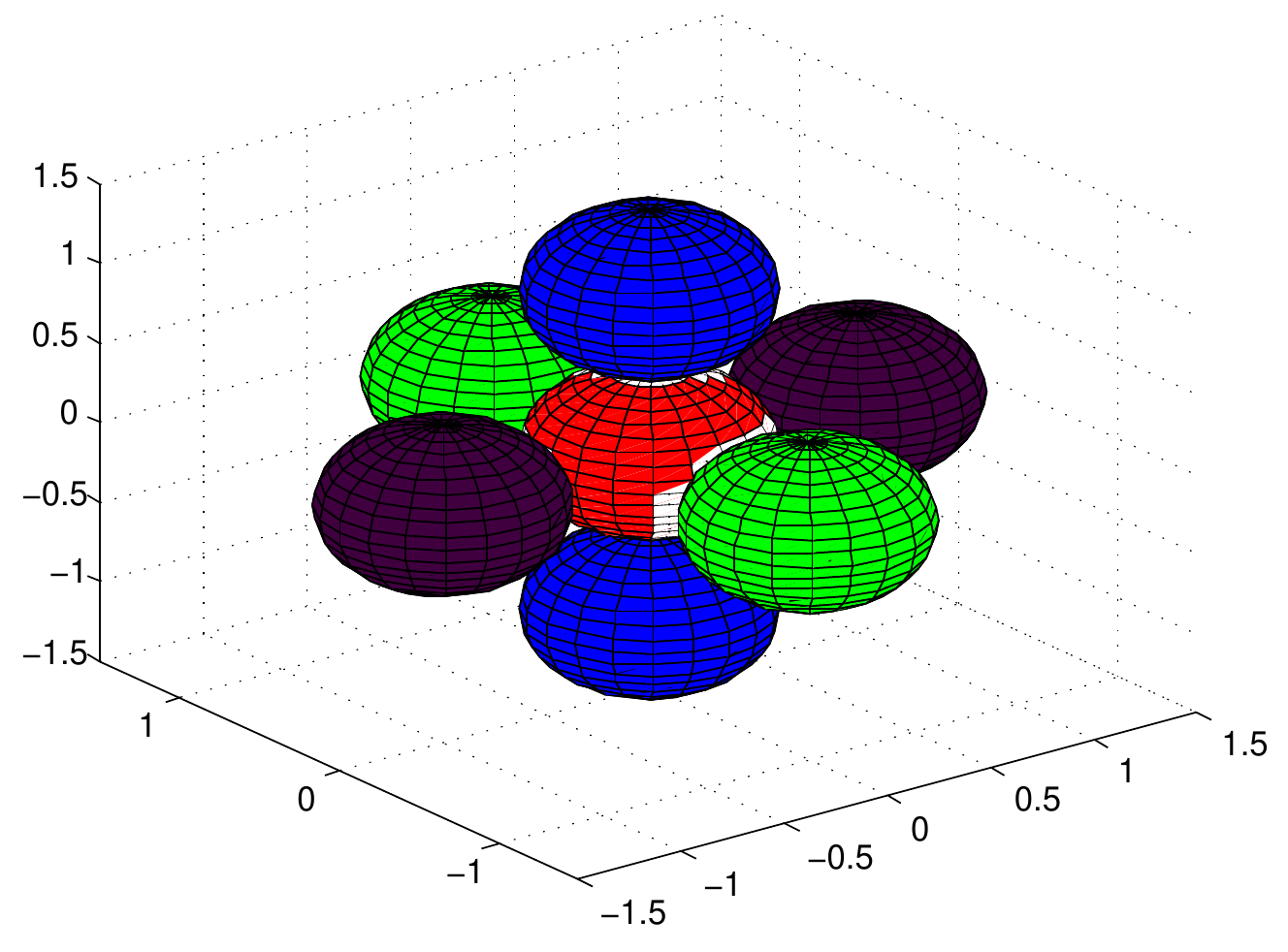} \quad
\includegraphics[scale=0.45]{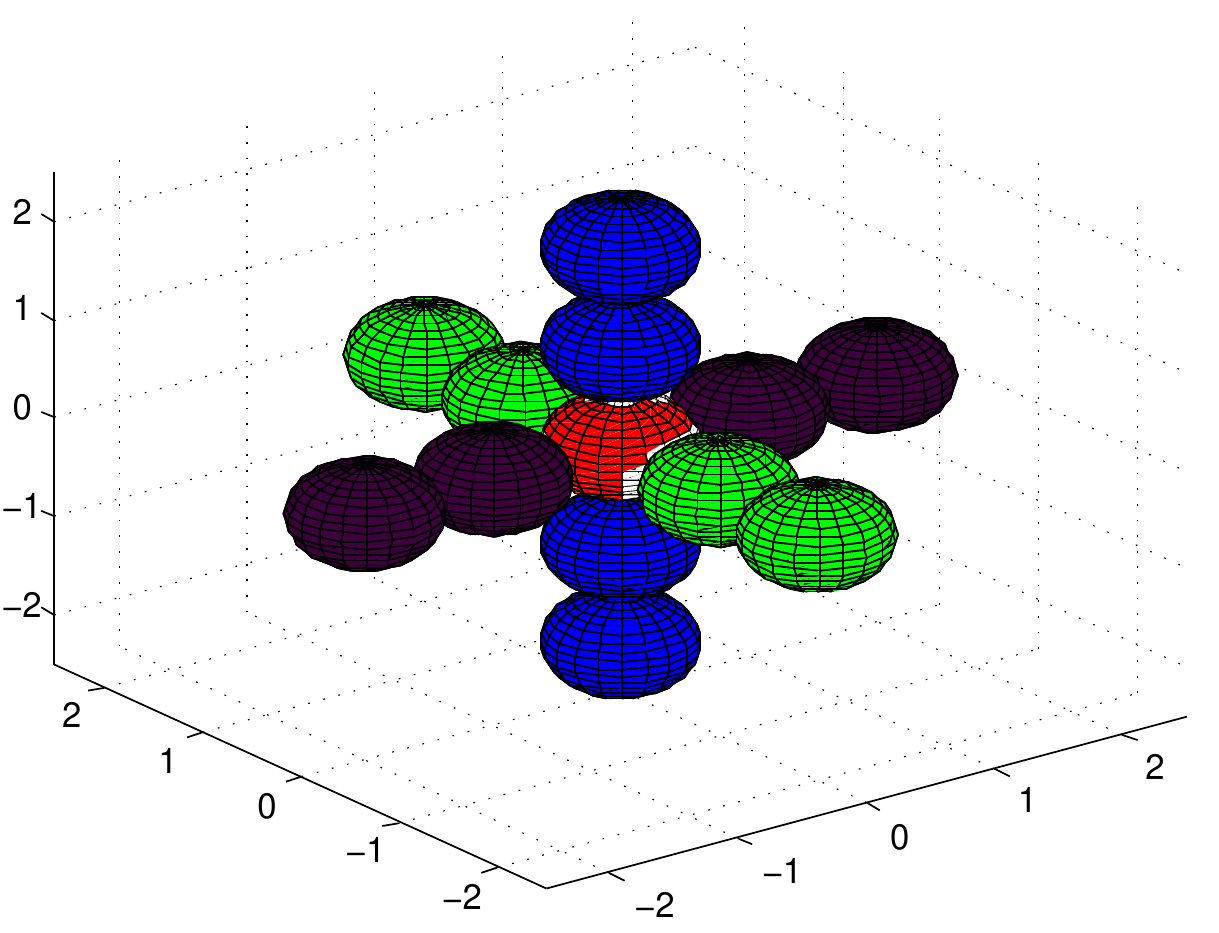}
\end{center}
\caption{The cells treated as the neighbors of centered cells. The left one considers nearest neighborhood, and the right one considers next-nearest neighborhood.}
\label{fig:neighborhood}
\end{figure}

In this paper, we consider the reversibility problem of $d$-dimensional linear cellular automata with the prolonged $\eta$-nearest neighborhood under null boundary conditions for $d \geq 3, \eta \in \mathbb{N}$, and characterize its inverse, if it exists. Figure \ref{fig:neighborhood} illustrates those cells that contribute to the evolution of the centered cells for $\eta = 1$ and $\eta = 2$. After revealing the matrix representation of a $d$-dimensional linear cellular automaton with the prolonged $\eta$-nearest neighborhood, studying the reversibility problem is equivalent to elaborating the invertibility of the corresponding matrix. We show that the associated matrix can be decomposed into the Kronecker sum of several smaller matrices, each of which is a Toeplitz matrix, and an algorithm for its inverse matrix, if it exists, is obtained. The main contributions of this paper are to significantly reduce the computational cost\footnote{Roughly speaking, the computational cost of characterizing the eigenvalues/eigenvectores/determinant of an $n \times n$ matrix is $O(n^3)$, and the computational cost of our approach is $O(n^{\frac{3}{d}})$, where $d$ is the dimension of the considered system.} of determining the reversibility of a multidimensional linear cellular automaton under null boundary condition and its reverse, if it exists, and the dynamical behavior of a multidimensional linear cellular automaton is characterized by the properties of block Toeplitz matrix (cf.~\cite{Gray-2006, GC-2012} for more details about the discussion of block Toeplitz matrix). Additionally, the study can extend to the ones under periodic boundary conditions with a minor modification.

The rest of this investigation is organized as follows. Section \ref{sec:preliminary} recalls definitions and fundamental results in matrix theory and number theory that are used in the later discussion. Sections \ref{sec:3d-ca-square-root} and \ref{sec:3d-ca-general} consider three-dimensional cellular automata with nearest neighborhood, and Section \ref{sec:multidimension-ca} extends the results to general multidimensional cellular automata with the prolonged $\eta$-nearest neighborhood. Conclusions and further discussion are given in Section \ref{sec:conclusion}.

\section{Preliminary} \label{sec:preliminary}

This section devotes to introducing the definition of three-dimensional linear cellular automata with nearest neighborhood over the finite field $\mathbb{Z}_p = \{0, 1, \ldots, p-1\}$ and recalling some well-established theorems in matrix theory and number theory.

Consider a three-dimensional infinite lattice $\mathbb{Z}^3$, divided in regular cells, the state of each cell is taking values from $\mathbb{Z}_p$. At each time step, the state of a cell changes according to a deterministic rule which depends on the state of each cell next to it; this is the so-called cellular automaton. Formally speaking, $\mathbb{Z}^3$ is the set of cells and the set $\mathbb{Z}_p^{\mathbb{Z}^3}$ is called the configuration space. For each $X \in \mathbb{Z}_p^{\mathbb{Z}^3}$ and $\mathbf{i} \in \mathbb{Z}^3$, $X_{\mathbf{i}} \in \mathbb{Z}_p$ refers to the state of $X$ at the cell $\mathbf{i}$.

Let $\{e_1, e_2, e_2\}$ denote the standard basis of $\mathbb{R}^3$; that is,
$$
e_1 = (1, 0, 0), \quad e_2 = (0, 1, 0), \quad \text{and} \quad e_3 = (0, 0, 1).
$$
Denote
$$
\mathcal{N} = \{\mathbf{0}, \pm e_1, \pm e_2, \pm e_3\}, \quad \text{where} \quad \mathbf{0} = (0,0,0).
$$
Fix $a, b, c, d, e, f, c_0 \in \mathbb{Z}_p$, define $\phi: \mathbb{Z}_p^{\mathcal{N}} \to \mathbb{Z}_p$ as
\begin{equation} \label{eq:3d-ca-local-rule}
\phi(y_{\mathcal{N}}) = c_0 y_{\mathbf{0}} + a y_{-e_2} + b y_{e_2} + c y_{-e_1} + d y_{e_1} + e y_{-e_3} + f y_{e_3} \pmod{p}
\end{equation}
A three-dimensional linear cellular automaton driven by the local rule $\phi$ with nearest neighborhood is defined as a pair $(\mathbb{Z}_p^{\mathbb{Z}^3}, \Phi)$, where $\Phi: \mathbb{Z}_p^{\mathbb{Z}^3} \to \mathbb{Z}_p^{\mathbb{Z}^3}$ is given by
\begin{align*}
\Phi(X)_{\mathbf{i}} &= \phi(X_{\mathbf{i}+\mathcal{N}}) \\
 &= c_0 X_{\mathbf{i}} + a X_{\mathbf{i}-e_2} + b X_{\mathbf{i}+e_2} + c X_{\mathbf{i}-e_1} + d X_{\mathbf{i}+e_1} + e X_{\mathbf{i}-e_3} + f X_{\mathbf{i}+e_3} \pmod{p}
\end{align*}
for each $\mathbf{i} \in \mathbb{Z}^3$. Namely, the state $X_{\mathbf{i}}(t+1)$ of cell $\mathbf{i}$ at time step $t+1$ is determined by
$$
X_{\mathbf{i}}(t+1) = \phi(X_{\mathbf{i}+\mathcal{N}}(t)), \quad t \in \mathbb{Z}^+.
$$
A cellular automaton under null boundary condition is one such that only finitely many cells are associated with nonzero state. More explicitly, for $n, s, m \in \mathbb{N}$, $n, s, m \geq 2$, a three-dimensional linear cellular automaton under null boundary condition is described as
\begin{align*}
X_{\mathbf{i}}(t+1) &= \Phi_N(X(t))_{\mathbf{i}} \\
 &= c_0 X_{\mathbf{i}}(t) + a(\mathbf{i}) X_{\mathbf{i}-e_2}(t) + b(\mathbf{i}) X_{\mathbf{i}+e_2}(t) + c(\mathbf{i}) X_{\mathbf{i}-e_1}(t) \\
 &\qquad + d(\mathbf{i}) X_{\mathbf{i}+e_1}(t) + e(\mathbf{i}) X_{\mathbf{i}-e_3}(t) + f(\mathbf{i}) X_{\mathbf{i}+e_3}(t) \pmod{p}
\end{align*}
where
\begin{align*}
c(\mathbf{i}) &= \left\{
\begin{aligned}
&c, && i_1 \geq 2; \\
&0, && i_1 = 1;
\end{aligned}\right. 
& a(\mathbf{i}) &= \left\{
\begin{aligned}
&a, && i_2 \geq 2; \\
&0, && i_2 = 1;
\end{aligned}\right. 
& e(\mathbf{i}) &= \left\{
\begin{aligned}
&e, && i_3 \geq 2; \\
&0, && i_3 = 1;
\end{aligned}\right. \\
d(\mathbf{i}) &= \left\{
\begin{aligned}
&b, && i_1 \leq n-1; \\
&0, && i_1 = n;
\end{aligned}\right. 
& b(\mathbf{i}) &= \left\{
\begin{aligned}
&d, && i_2 \leq s-1; \\
&0, && i_2 = s;
\end{aligned}\right. 
& f(\mathbf{i}) &= \left\{
\begin{aligned}
&f, && i_3 \leq m-1; \\
&0, && i_3 = m;
\end{aligned}\right. \quad
\end{align*}
and $\mathbf{i} = (i_1, i_2, i_3)$. That is, $\Phi_N: \mathbb{Z}^{n \times s \times m} \to \mathbb{Z}^{n \times s \times m}$ illustrates the evolution of each cell in an $n \times s \times m$ cuboid.

It is well-known that elaborating the behavior of a linear dynamical system is related to studying its corresponding matrix representation. Before characterizing the matrix representation of $\Phi_N$, which is postponed to the following section, we recall some definitions and results in matrix theory first.

\begin{definition}
Let $A \in \mathcal{M}_{j_1 \times j_2}(\mathbb{R})$ be a $j_1 \times j_2$ real matrix and $B \in \mathcal{M}_{k_1 \times k_2}(\mathbb{R})$ a $k_1 \times k_2$ real matrix. The \textbf{Kronecker product} (or \textbf{tensor product}) of $A$ and $B$ is the $j_1 k_1 \times j_2 k_2$ matrix
$$
A \otimes B = \begin{pmatrix}
a_{11} B & a_{12} B & \cdots & a_{1 j_2} B \\
\vdots & \vdots & & \vdots \\
a_{j_1 1} B & a_{j_1 2} B & \cdots & a_{j_1 j_2} B
\end{pmatrix}_{j_1 k_1 \times j_2 k_2}.
$$
\end{definition}

Kronecker products have many interesting properties. For example,
$$
(A \otimes B) (C \otimes D) = (A C) \otimes (B D)
$$
and
$$
\mathrm{rank}(A \otimes B) = \mathrm{rank}(A) \mathrm{rank}(B)
$$
for all matrices $A, B, C$, and $D$ provided that the products $AC$ and $BD$ are both well-defined.

Let $A$ be a $j \times j$ matrix and $B$ a $k \times k$ matrix. Denote $I_r$ the $r \times r$ identity matrix; the \textbf{Kronecker sum} (or \textbf{tensor sum}) of $A$ and $B$ is the $jk \times jk$ matrix $(I_k \otimes A) + (B \otimes I_j)$. The Kronecker product and Kronecker sum of matrices over finite fields are defined in the same way. From the definitions of Kronecker product and Kronecker sum, a straightforward examination reveals that the eigenvalues and eigenvectors of the Kronecker sum of $A$ and $B$ are completely illustrated by the eigenvalues and eigenvectors of $A$ and $B$, as follows.

Suppose that $\mu_1, \mu_2, \ldots, \mu_j$ are the eigenvalues of $A$ and $\nu_1, \nu_2, \ldots, \nu_k$ are the eigenvalues of $B$. Then $\{\mu_i + \nu_r\}_{1 \leq i \leq j, 1 \leq r \leq k}$ is the set of eigenvalues of $(I_k \otimes A) + (B \otimes I_j)$. Furthermore, if $\mathbf{x}$ is an eigenvector of $A$ corresponding to the eigenvalue $\mu$ and $\mathbf{y}$ is an eigenvector of $B$ corresponding to the eigenvalue $\nu$, then $\mathbf{y} \otimes \mathbf{x}$ is an eigenvector of $(I_k \otimes A) + (B \otimes I_j)$ corresponding to the eigenvalue $\mu + \nu$. These results hold for matrices either over $\mathbb{R}$ or finite field $\mathbb{F}_p$; the proofs of finite field case are analogous to the original ones, hence are omitted for the compactness of this paper. The reader is referred to \cite{HJ-1994,Orteg-1987} for more details.

It is known that a given matrix $A$ is reversible if and only if $\det A \neq 0$, and $\det A$ is the product of its eigenvalues. Whenever $A$ is decomposed into the Kronecker sum of small matrices $B$ and $C$, it follows that $A$ is reversible only if either $B$ or $C$ is reversible. To reveal the necessary and sufficient condition of $A$ being reversible, we aim to characterize the eigenvalues of $B$ and $C$ completely. Completely characterizing the eigenvalues of $B$ and $C$ not only reduces the computational cost, but helps in determining the inverse matrix of $A$, if it exists. The related discussion is addressed later.

Next, we recall some results in number theory that will be used later for the investigation of eigenvalues. Suppose that $\mathbb{F}_p$ is a finite field with characteristic $p$ and $F(x)$ is a polynomial over $\mathbb{F}_p$. Let $\alpha$ be a root of $F(x)$; the \textbf{multiplicity} of $\alpha$ is the largest positive integer $n$ for which $(x-\alpha)^n$ divides $F(x)$. $\alpha$ is a \textbf{simple root} if $n = 1$ and is a \textbf{multiple root} otherwise.

\begin{definition}
Suppose that $F(x) \in \mathbb{F}_p [x]$ is irreducible. We say that $F(x)$ is \textbf{separable} if it has no multiple roots in any extension of $\mathbb{F}_p$. An irreducible polynomial that is not separable is \textbf{inseparable}.
\end{definition}

\begin{proposition}[See \cite{Rom-2006}]\label{prop:irreducible-polynomial-only-simple-roots}
All irreducible polynomials over $\mathbb{F}_p$ are separable.
\end{proposition}

When the degree of $F(x)$ is $2$, it is seen that $x^2 - u$ is irreducible if and only if $u \notin \mathbb{F}_p^2$ (\cite{MM-2007}). A field $\mathbb{E}$ is a finite extension of $\mathbb{F}_p$ if $\mathbb{F}_p \subseteq \mathbb{E}$ and $\mathbb{E}$ is a finite dimensional vector space over $\mathbb{F}_p$. If $F(x) \in \mathbb{F}_p[x]$ factors into linear factors
$$
F(x) = \alpha_0 (x - \alpha_1) (x - \alpha_2) \cdots (x - \alpha_k)
$$
in an extension field $\mathbb{E}$, we say that $F(x)$ \textbf{splits} in $\mathbb{E}$.

\begin{definition}
Let $\mathcal{F} = \{F_i(x): i \in I\} \subseteq \mathbb{F}_p[x]$ be a family of polynomials. A \textbf{splitting field} for $\mathcal{F}$ is an extension field $\mathbb{E}$ of $\mathbb{F}_p$ satisfying:
\begin{enumerate}[\bf 1)]
\item Each $F_i(x)$ splits over $\mathbb{E}$;
\item $\mathbb{E}$ is the smallest field that contains the roots of each $F_i(x)$.
\end{enumerate}
\end{definition}

The next theorem describes an important property of irreducible polynomials over $\mathbb{F}_p$.

\begin{theorem}[See \cite{MM-2007}]
If $F(x) \in \mathbb{F}_p[x]$ is irreducible of degree $k$, then $F$ has a root $\alpha$ in an extension field $\mathbb{E}$ of $\mathbb{F}_p$ with cardinality $p^k$. Moreover, all the roots of $F(x)$ are simple and are given by $\alpha, \alpha^p, \ldots, \alpha^{p^{k-1}}$.
\end{theorem}

For more details, the reader is referred to \cite{MM-2007,Rom-2006}.

\section{Three-Dimensional Cellular Automata: Reciprocal Cases} \label{sec:3d-ca-square-root}

Beyond determining whether a multidimensional linear cellular automaton under the null boundary condition is reversible, we aim to propose a method to find the reverse of an invertible cellular automaton. The rough idea is as follows.
\begin{enumerate}[(i)]
\item Decompose the matrix representation of the cellular automaton into the Kronecker sum of several components.
\item Find the Jordan normal forms of these smaller matrices, respectively.
\item The reversibility of the original system comes from the combination of eigenvalues of these components, so does its inverse.
\end{enumerate}

In the following two sections, we introduce the matrix representation of the three-dimensional linear cellular automata $\Phi_N$ over $\mathbb{Z}_p$ under null boundary conditions with local rule $\phi$ defined in \eqref{eq:3d-ca-local-rule} and elaborate the reversibility of $\Phi_N$ via its matrix representation. We start the discussion with the three dimensional since there is something interesting which might not be seen in dimension two. More precisely, we can have two eigenvalues that sum to the negation of a third one, which leads to the irreversibility of the elaborated system.

This section considers the case where the evolution at each cell is independent of its current state (i.e., $c_0 = 0$) and each pair of parameters in every direction satisfies the \textbf{quadratic reciprocity law} (defined later). The general cases are elucidated in Section \ref{sec:3d-ca-general}.

Fix $n, s, m \geq 2$, define
$$
T_{RN} = \begin{pmatrix}
  M_s & e I_{ns} & O_{ns} & \cdots & O_{ns} & O_{ns} \\
  f I_{ns} & M_s & e I_{ns} & \cdots & O_{ns} & O_{ns} \\
  O_{ns} & f I_{ns} & M_s & \cdots & O_{ns} & O_{ns} \\
  \vdots & \vdots & \vdots & \ddots & \vdots & \vdots \\
  O_{ns} & O_{ns} & O_{ns} & \cdots & M_s & e I_{ns} \\
  O_{ns} & O_{ns} & O_{ns} & \cdots & f I_{ns} & M_s
\end{pmatrix}_{nsm \times nsm},
$$
where
$$
M_s = \begin{pmatrix}
  S_n(c, d) & b I_n & O_n & \cdots & O_n & O_n \\
  a I_n & S_n(c, d) & b I_n & \cdots & O_n & O_n \\
  O_n & a I_n & S_n(c, d) & \cdots & O_n & O_n \\
  \vdots & \vdots & \vdots & \ddots & \vdots & \vdots \\
  O_n & O_n & O_n & \cdots & S_n(c, d) & b I_n \\
  O_n & O_n & O_n & \cdots & a I_n & S_n(c, d)
\end{pmatrix}_{ns \times ns}
$$
with
$$
S_n(c, d) = \begin{pmatrix}
  0 & d & 0 & \cdots & 0 & 0 \\
  c & 0 & d & \cdots & 0 & 0 \\
  0 & c & 0 & \cdots & 0 & 0 \\
  \vdots & \vdots & \vdots & \ddots & \vdots & \vdots \\
  0 & 0 & 0 & \cdots & 0 & d \\
  0 & 0 & 0 & \cdots & c & 0
\end{pmatrix}_{n \times n};
$$
herein, $O_k$ refers to the $k \times k$ zero matrix. We remark that $M_s$ is the Kronecker sum of $S_n(c, d)$ and $S_s(a, b)$, and $T_{RN}$ is the Kronecker sum of $M_s$ and $S_m(f, e)$. Let $\Theta: \mathbb{Z}_p^{n \times s \times m} \to \mathbb{Z}_p^{nsm}$ be the transformation that designates $X = (X_{\mathbf{i}})_{1 \leq i_1 \leq n, 1 \leq i_2 \leq s, 1 \leq i_3 \leq m}$ as a column vector with respect to the anti-lexicographic order, where $\mathbf{i} = (i_1, i_2, i_3)$. For example, if $n = s = m = 2$, then
$$
\Theta(X) = (X_{111}, X_{211}, X_{121}, X_{221}, X_{112}, X_{212}, X_{122}, X_{222})',
$$
where $v'$ denotes the transpose of $v$. Notably, $\Theta$ is a one-to-one correspondence. The following theorem indicates that $T_{RN}$ is the matrix representation of the cellular automaton $\Phi_N$ with local rule $\phi$ defined in \eqref{eq:3d-ca-local-rule} and $c_0 = 0$.

\begin{theorem} \label{thm:3d-matrix-representation-commute-diagram}
The three-dimensional linear cellular automaton $\Phi_N$ over $\mathbb{Z}_p$ under null boundary condition is characterized by $T_{RN}$, and vice versa. More explicitly, the diagram
$$
\xymatrix{
\mathbb{Z}_p^{n \times s \times m} \ar[rr]^{\Phi_N} \ar[d]_{\Theta} && \mathbb{Z}_p^{n \times s \times m} \ar[d]^{\Theta} \\
\mathbb{Z}_p^{nsm} \ar[rr]_{\mathbf{T}} && \mathbb{Z}_p^{nsm}}
$$
commutes, where $\mathbf{T}(y) = T_{RN} y \pmod{p}$ for each $y \in \mathbb{Z}^{nsm}$.
\end{theorem}
\begin{proof}
The verification is straightforward, thus it is omitted.
\end{proof}

Theorem \ref{thm:3d-matrix-representation-commute-diagram} reveals that $\Phi_N$ is reversible if and only if its matrix representation $T_{RN}$ is invertible over $\mathbb{Z}_p$; this initiates the study of the eigenvalues of $T_{RN}$.

For $j \geq 0$, define
\begin{equation}\label{eq:charpoly-for-Kj}
g_j(x) = \sum\limits_{i=0}^{[j/2]} (-1)^i {{j-i}\choose{i}} x^{j-2i},
\end{equation}
herein $[\cdot]$ is the floor function.

For any collection of sets $\{S_i\}_{i=1}^k$, let $S_1 + S_2 + \cdots + S_k$ denote the Minkowski sum of sets; that is,
$$
S_1 + S_2 + \cdots + S_k = \{s_1 + s_2 + \cdots + s_k: s_i \in S_i \text{ for } 1 \leq i \leq k\}.
$$
Furthermore, we refer to $S_r(1,1)$ as $K_r$ for the sake of simplicity. Theorem \ref{thm:eigenvalue-set-decomposition} characterizes all the eigenvalues of $T_{RN}$ completely. Herein, we consider the general cases (i.e., real coefficients) first, the case of coefficients in $\mathbb{Z}_p$ is elaborated later on.

We remark that the discussion of eigenvalues of $T_{RN}$ over the extension field of $\mathbb{Z}_p$ is analogous to the real matrix case (over complex numbers). The reader is referred to Examples \ref{eg:m=n=s=4-p=3}, \ref{eg:m=n=s=2-p-not3}, and \ref{eg:m=n=s=4-p=5-semiJordan} for more details. Furthermore, the investigation can extend to cellular automata defined on an arbitrary field $\mathbb{F}$.

\begin{theorem}\label{thm:eigenvalue-set-decomposition}
Suppose that $a, b, c, d, e, f \in \mathbb{R}^+$. The set of eigenvalues $\mathbf{E}_T$ of $T_{RN}$ is
$$
\mathbf{E}_T = \alpha_{a, b} \mathbf{R}_s + \alpha_{c, d} \mathbf{R}_n + \alpha_{e, f} \mathbf{R}_m,
$$
where $\mathbf{R}_j$ denotes the set of roots of $g_j$ over $\mathbb{C}$ and $\alpha_{t_1, t_2}$ denotes the geometric mean of $t_1, t_2$ in $\mathbb{R}$ provided that it is well-defined.
\end{theorem}
\begin{proof}
Suppose that $t_1, t_2 \in \mathbb{R}^+$ are positive real numbers, and $A$ is an $r \times r$ real matrix. Given a $qr \times qr$ matrix
$$
\widehat{A} = \begin{pmatrix}
A & t_1 I_r & O_r & O_r & \cdots & O_r \\
t_2 I_r & A & t_1 I_r & O_r & \cdots & O_r \\
O_r & t_2 I_r & A & t_1 I_r & \cdots & O_r \\
\vdots & \vdots & \ddots & \ddots & \ddots & \vdots \\
O_r & O_r & \cdots & t_2 I_r & A & t_1 I_r \\
O_r & O_r & \cdots & O_r & t_2 I_r & A
\end{pmatrix}.
$$
Let $P = \mathrm{diag}(\alpha_{t_1^{-1},t_2} I_r, \alpha_{t_1^{-1},t_2}^2 I_r, \cdots, \alpha_{t_1^{-1},t_2}^q I_r)$ be a $qr \times qr$ invertible matrix with diagonal block $\alpha_{t_1^{-1},t_2}^i I_r$ for $1 \leq i \leq q$. A straightforward examination infers that
$$
P^{-1} \widehat{A} P = \begin{pmatrix}
A & \alpha_{t_1,t_2} I_r & O_r & O_r & \cdots & O_r \\
\alpha_{t_1,t_2} I_r & A & \alpha_{t_1,t_2} I_r & O_r & \cdots & O_r \\
O_r & \alpha_{t_1,t_2} I_r & A & \alpha_{t_1,t_2} I_r & \cdots & O_r \\
\vdots & \vdots & \ddots & \ddots & \ddots & \vdots \\
O_r & O_r & \cdots & \alpha_{t_1,t_2} I_r & A & \alpha_{t_1,t_2} I_r \\
O_r & O_r & \cdots & O_r & \alpha_{t_1,t_2} I_r & A
\end{pmatrix}.
$$
Notably, $P^{-1} \widehat{A} P = I_q \otimes A + (\alpha_{t_1,t_2} K_q) \otimes I_r$.

Define $P_{T_{RN}}, P_{M_s}$, and $P_{S_n}$ as
\begin{align*}
P_{T_{RN}} &= \mathrm{diag}(\alpha_{e^{-1},f} I_{ns}, \alpha_{e^{-1},f}^2 I_{ns}, \cdots, \alpha_{e^{-1},f}^m I_{ns}), \\
P_{M_s} &= \mathrm{diag}(\alpha_{a,b^{-1}} I_{n}, \alpha_{a,b^{-1}}^2 I_{n}, \cdots, \alpha_{a,b^{-1}}^s I_{n}), \\
P_{S_n} &= \mathrm{diag}(\alpha_{c,d^{-1}}, \alpha_{c,d^{-1}}^2, \cdots, \alpha_{c,d^{-1}}^n),
\end{align*}
respectively. It follows immediately that
\begin{align*}
P_{T_{RN}}^{-1} T_{RN} P_{T_{RN}} &= I_m \otimes M_s + (\alpha_{e,f} K_m) \otimes I_{ns}, \\
P_{M_s}^{-1} M_s P_{M_s} &= I_s \otimes S_n(c, d) + (\alpha_{a,b} K_s) \otimes I_n,
\end{align*}
and $P_{S_n}^{-1} S_n(c, d) P_{S_n} = \alpha_{c,d} K_n$. Let $\mathbf{E}_j$ denote the set of eigenvalues of $K_j$ for $j \in \mathbb{N}$. We then derive that $\mathbf{E}_T = \alpha_{a, b} \mathbf{E}_s + \alpha_{c,d} \mathbf{E}_n + \alpha_{e, f} \mathbf{E}_m$. It remains to show that $\mathbf{E}_j = \mathbf{R}_j$.

Let $g_j (x) = \det (x I_j - K_j)$ be the characteristic polynomial of $K_j$ for $j \geq 1$. Let $g_0(x) = 1$ and $g_j(x) = 0$ for $j < 0$. It can be verified that $g_j(x)$ satisfies the following recurrence relation:
$$
g_j (x) = x g_{j-1} (x) - g_{j-2} (x), \qquad j \geq 1.
$$
Let $G(u, x) = \sum\limits_{j \geq 0} g_j(x) u^j$ be the generating function. It follows immediately that
$$
G(u, x) = \dfrac{1}{u^2 - xu + 1} = \sum_{j \geq 0} (u (x - u))^j.
$$
We can conclude that
$$
g_j(x) = \sum\limits_{i=0}^{[j/2]} (-1)^i {{j-i}\choose{i}} x^{j-2i}.
$$
This completes the proof.
\end{proof}

We remark that, when the discussion focuses on the finite field $\mathbb{Z}_p$, some of the eigenvalues of $T_{RN}$ are in the algebraic closure of $\mathbb{Z}_p$ just like the real matrices considered in Theorem \ref{thm:eigenvalue-set-decomposition}. Therefore, we study the eigenvalues of $T_{RN}$ in the extension field of $\mathbb{Z}_p$ where the characteristic polynomial of $T_{RN}$ splits in. Furthermore, the discussion of Theorems \ref{thm:eigenvalue-set-decomposition} and \ref{thm:main-theorem-algorithm-for-JT-and-Inverse} apply to real matrices analogously.

\begin{proposition}\label{prop:roots-of-gj}
Suppose that $k < \ell$ are two positive integers. Let $h(x) = \mathrm{gcd}(g_k(x), g_{\ell}(x)) \in \mathbb{Z}[x]$ be the greatest common divisor of $g_k(x)$ and $g_{\ell}(x)$. Then
\begin{enumerate}[\bf (i)]
\item $\mathrm{deg}~h(x) = \mathrm{gcd}(k+1, \ell + 1) - 1$;
\item $h(x) = g_k(x)$ if and only if $(k+1) | (\ell + 1)$.
\end{enumerate}
\end{proposition}
\begin{proof}
The proof of Theorem \ref{thm:eigenvalue-set-decomposition} reveals the generating function $G(u, x)$ of $g_j(x)$ as
$$
G(u, x) = \dfrac{1}{u^2 - xu + 1} = \dfrac{1}{\sqrt{x^2-4}} \left(\dfrac{1}{u - \frac{x + \sqrt{x^2-4}}{2}} - \dfrac{1}{u - \frac{x - \sqrt{x^2-4}}{2}}\right).
$$
This demonstrates that an alternative expression of $g_j(x)$ is
$$
g_j(x) = \dfrac{1}{\sqrt{x^2-4}} \left[\left(\dfrac{x + \sqrt{x^2-4}}{2}\right)^{j+1} - \left(\dfrac{x - \sqrt{x^2-4}}{2}\right)^{j+1}\right].
$$
Suppose that $\lambda$ is a root of $g_j(x)$. Then
$$
\left(\frac{\lambda + \sqrt{\lambda^2-4}}{\lambda - \sqrt{\lambda^2-4}}\right)^{j+1} = 1 \quad \Rightarrow \quad \dfrac{\lambda + \sqrt{\lambda^2-4}}{\lambda - \sqrt{\lambda^2-4}} = \exp \left( \frac{2r\pi}{j+1}i \right),
$$
where $r = 1, 2, \cdots, j$. It follows that
$$
\left(\lambda + \sqrt{\lambda^2-4}\right)^2 = 4 \exp \left( \frac{2r\pi}{j+1}i \right)
$$
Thus,
$$
\lambda + \sqrt{\lambda^2-4} = \pm 2 \exp \left( \frac{r\pi}{j+1}i \right) \quad \Rightarrow \quad \lambda = \pm 2 \cos \frac{r\pi}{j+1}
$$
for $1 \leq r \leq j$. Since $\cos \dfrac{r \pi}{j+1} = - \cos \dfrac{j+1-r}{j+1}\pi$, we have derived that
$$
\lambda = 2 \cos \frac{r\pi}{j+1}, \qquad 1 \leq r \leq j.
$$

Let $\{\lambda_r\}_{r=1}^k$ and $\{\lambda'_q\}_{q=1}^{\ell}$ be the set of roots of $g_k(x)$ and $g_{\ell}(x)$, respectively. It is seen that $\lambda_r = \lambda'_q$ for some $q, r$ if and only if $\dfrac{r}{k+1} = \dfrac{q}{\ell+1}$. This demonstrates that $\mathrm{deg}~h(x) = \mathrm{gcd}(k+1, \ell+1) - 1$ and $h(x) = g_k(x)$ if and only if $(k+1) | (\ell + 1)$.
\end{proof}

To investigate the reversibility of $T_{RN}$, Theorem \ref{thm:eigenvalue-set-decomposition} infers that it is essential to consider the case where the geometric means of the three pairs of parameters $\{a, b\}, \{c, d\}$, and $\{e, f\}$ exist, respectively. Let $\mathbb{Z}_p^* = \mathbb{Z}_p \setminus \{0\}$ and let $H = \{\ell \in \mathbb{Z}_p^*: \ell^{(p-1)/2} \equiv 1 \pmod{p}\}$. A proper subset $K \subset \mathbb{Z}_p^*$ is said to be in the \textbf{same partition} if $K \subseteq H$ or $K \subseteq H^c$. Euler's criterion demonstrates that $\ell \in H$ if and only if $x^2 \equiv \ell \pmod{p}$ for some $x \in \mathbb{Z}_p^*$. Suppose that $\{a, b\}, \{c, d\}$, and $\{e, f\}$ are in the same partition, respectively. It follows that the geometric means of $\{a, b\}, \{c, d\}$, and $\{e, f\}$ are in $\mathbb{Z}_p$, respectively, since $|H| = |H^c| = \dfrac{p-1}{2}$. Therefore, Theorem \ref{thm:eigenvalue-set-decomposition} still holds over the finite field $\mathbb{Z}_p$.

Define $k: \mathbb{Z}_p \times \mathbb{Z}_p \to \mathbb{Z}_p$ as
$$
k(t_1, t_2) = \left\{
\begin{aligned}
&\min\limits_{0 \leq t \leq p-1} \{t: t^2 = t_1 t_2\}, & & t_1 \neq t_2; \\
&t_1, & & t_1 = t_2.
\end{aligned}
\right.
$$
Notably, $k(t_1, t_2)$ is well-defined if $t_1 = t_2$ or the pair $\{t_1, t_2\}$ is in the same partition. For each prime $p$, let $\mathbf{SR}_p \subset \mathbb{Z}_p \times \mathbb{Z}_p$ be the domain of $k(t_1, t_2)$; we say that $(t_1, t_2) \in \mathbb{Z}_p \times \mathbb{Z}_p$ satisfies the \textbf{quadratic reciprocity law} if $(t_1, t_2) \in \mathbf{SR}_p$. For the simplicity of the notations, we denote $k(a, b), k(c, d)$, and $k(e, f)$ by $k_s, k_n$, and $k_m$, respectively.

It is known that $T_{RN}$ is reversible if and only if $0$ is not an eigenvalue of $T_{RN}$. The proof of Proposition \ref{prop:roots-of-gj} and some numerical experiments suggest that $T_{RN}$ is reversible over $\mathbb{R}$ if $m+1, n+1, s+1$ are pairwise relatively prime and $\alpha_{a, b}, \alpha_{c, d}, \alpha_{e, f} \in \mathbb{Q}$. However, the discussion of the reversibility of $T_{RN}$ over $\mathbb{Z}_p$ is more complicated since it happens that $\mathrm{deg} (\mathrm{gcd}(g_k^{[p]}(x), g_{\ell}^{[p]}(x))) > \mathrm{deg} (\mathrm{gcd}(g_k(x), g_{\ell}(x)))$ for some $p$ and $k, \ell$, where $F^{[p]}(x)$ is defined as $F^{[p]}(x) \equiv F(x) \pmod{p}$. For instance, $g_4(x)$ and $g_{78}(x)$ are relatively prime in $\mathbb{Z}[x]$ while $\mathrm{gcd}(g^{[3]}_4(x), g^{[3]}_{78}(x)) \equiv t^4 + 1 \pmod{3}$. On the other hand, numerical experiments indicate that $\mathrm{deg} (\mathrm{gcd}(g_k^{[p]}(x), g_{\ell}^{[p]}(x))) = \mathrm{deg} (\mathrm{gcd}(g_k(x), g_{\ell}(x)))$ for $p <100$ and $k, \ell < 77$.

For the rest of this section, we assume that $(a, b), (c, d)$, and $(e, f)$ satisfy the quadratic reciprocity law until stated otherwise.

\begin{corollary}\label{cor:m=n=s=2-invertible-condition}
If $n=s=m=2$, then $T_{RN}$ is reversible if and only if
\begin{enumerate}[\bf i)]
\item $k_1 \equiv k_2 \equiv k_3 \pmod{p}$ and $p \neq 3$.
\item $k_1 \equiv k_2 \not\equiv k_3$ and $k_3 \not\equiv \pm 2 k_1 \pmod{p}$.
\item $k_1, k_2$, and $k_3$ are pairwise distinct and $k_3 \not\equiv \pm (k_1 \pm k_2) \pmod{p}$.
\end{enumerate}
Herein $k_1, k_2, k_3 \in \{k_m, k_n, k_s\}$.
\end{corollary}
\begin{proof}
Since $g_2^{[p]}(x) \equiv (x+1)(x-1) \pmod{p}$, the desired result follows from straightforward examination.
\end{proof}

\begin{corollary}\label{cor:m=n=s=2-inverse}
Suppose that $n=s=m=2$ and $T_{RN}$ is reversible. Then
$$
T_{RN}^{-1} \sim \mathrm{diag}(1, 1, 1, p-1, p-1, p-1, 3^{-1}, p - 3^{-1}) \pmod{p},
$$
where $A \sim B$ denotes that the matrices $A$ and $B$ are similar.
\end{corollary}
Corollary \ref{cor:m=n=s=2-inverse} is an immediate application of Theorem \ref{thm:eigenvalue-set-decomposition}. The detailed discussion is postponed to Example \ref{eg:m=n=s=2-p-not3}.

Proposition \ref{prop:irreducible-imply-diagonalizable} comes immediately from Theorem \ref{thm:eigenvalue-set-decomposition} and Proposition \ref{prop:irreducible-polynomial-only-simple-roots}, the proof is thus omitted.

\begin{proposition}\label{prop:irreducible-imply-diagonalizable}
If $g_n^{[p]}, g_s^{[p]}, g_m^{[p]}$ are irreducible over $\mathbb{Z}_p$, then
$$
T_{RN} \sim \mathrm{diag}(k_n \lambda_i + k_s \kappa_j + k_m \iota_{\ell})_{1 \leq i \leq n, 1 \leq j \leq s, 1 \leq \ell \leq m}.
$$
Furthermore, if $T_{RN}$ is reversible, then
$$
T_{RN}^{-1} \sim \mathrm{diag}((k_n \lambda_i + k_s \kappa_j + k_m \iota_{\ell})^{-1})_{1 \leq i \leq n, 1 \leq j \leq s, 1 \leq \ell \leq m},
$$
where $\{\lambda_i\}_{i=1}^n, \{\kappa_j\}_{j=1}^s$, and $\{\iota_{\ell}\}_{\ell=1}^m$ are roots of $g_n^{[p]}, g_s^{[p]}$, and $g_m^{[p]}$ over the splitting field for $\{g_n^{[p]}, g_s^{[p]}, g_m^{[p]}\}$.
\end{proposition}

\begin{example}\label{eg:m=n=s=4-p=3}
Suppose that $n = s = m = 4$ and $p = 3$. Since $g_4(x) = x^4 - 3x + 1 \equiv (x^2 + x + 2)(x^2 + 2x + 2) \pmod{3}$ is decomposed as two relatively prime irreducible polynomials, we can conclude that two factors of $g_4(x)$ are separable in their splitting field $\mathbb{Z}_3(\alpha)$, where $\alpha$ is a root of $x^2 + 1 \equiv 0 \pmod{3}$. It can be verified that $k_n = k_s = k_m = 1$. Proposition \ref{prop:irreducible-imply-diagonalizable} infers that
$$
T_{RN} \sim \mathrm{diag}(\lambda_i + \lambda_j + \lambda_k)_{1 \leq i, j, k \leq 4},
$$
where $\lambda_1 = 1 + \alpha, \lambda_2 = 1 + 2 \alpha$ are the roots of $x^2 + x + 2$, and $\lambda_3 = 2 + \alpha, \lambda_4 = 2 + 2 \alpha$ are the roots of $x^2 + 2x + 2$. Since $0$ is an eigenvalue of $T_{RN}$, the cellular automaton $\Phi$ is irreversible.
\end{example}

\begin{lemma}\label{lem:Inverse-of-Jordan-form}
Suppose that $\{A_i\}_{i=1}^k$ are a collection of invertible $r \times r$ matrices. Then
\begin{equation}\label{eq:semi-Jordan-form}
\mathbf{A} = \begin{pmatrix}
A_1 & \omega_1 I_r & O_r & \cdots & O_r \\
O_r & A_2 &  \omega_2 I_r & \cdots & O_r \\
\vdots & \vdots & \ddots & \vdots & \vdots \\
O_r & O_r & \cdots & A_{k-1} & \omega_{k-1} I_r \\
O_r & O_r & \cdots & O_r & A_k
\end{pmatrix}
\end{equation}
is invertible with inverse matrix
$$
\mathbf{A}^{-1} = \begin{pmatrix}
A_1^{-1} & - \omega_1 A_1^{-1}A_2^{-1} & \omega_1 \omega_2 A_1^{-1}A_2^{-1}A_3^{-1} & \cdots & (-1)^{k-1} \prod\limits_{i=1}^{k-1} \omega_i \prod\limits_{i=1}^k A_i^{-1} \\
O_r & A_2^{-1} & - \omega_2 A_2^{-1}A_3^{-1} & \cdots & (-1)^{k-2} \prod\limits_{i=2}^{k-1} \omega_i \prod\limits_{i=2}^k A_i^{-1} \\
\vdots & \vdots & \ddots & \vdots & \vdots \\
O_r & O_r & \cdots & A_{k-1}^{-1} & - \omega_{k-1} A_{k-1}^{-1} A_k^{-1} \\
O_r & O_r & \cdots & O_r & A_k^{-1}
\end{pmatrix},
$$
where $\omega_i \in \mathbb{R}$ for $1 \leq i \leq k-1$.
\end{lemma}
\begin{proof}
The proof is straightforward, and thus it is omitted.
\end{proof}

In matrix theory, Jordan form of a given matrix reveals the most important and essential information about it, such as the reversibility and limiting behavior. A matrix of the form \eqref{eq:semi-Jordan-form} is called a \textbf{generalized Jordan form} if $\omega_i \in \{0, 1\}$ for $1 \leq i \leq k-1$. In the rest of this paper, the classical Jordan form is called canonical Jordan form to distinguish it from the generalized cases.

It is seen that the generalized Jordan form is the block-type canonical Jordan form, and $T_{RN}$ is itself a generalized Jordan form if $e = 1$ and $f = 0$. Lemma \ref{lem:Inverse-of-Jordan-form} indicates that the generalized Jordan form helps in determining whether a matrix is reversible and characterizing its inverse whenever it exists.

\begin{theorem}\label{thm:main-theorem-algorithm-for-JT-and-Inverse}
Suppose that $(a, b), (c, d), (e, f)$ satisfy the quadratic reciprocity law. There is an algorithm for the computation of the generalized Jordan form of $T_{RN}$ and $T_{RN}^{-1}$, if it exists, over the splitting field for $g_n^{[p]}(x), g_s^{[p]}(x)$, and $g_m^{[p]}(x)$.
\end{theorem}

The proof of Theorem \ref{thm:main-theorem-algorithm-for-JT-and-Inverse} is divided into several parts. Firstly, we decompose $T_{RN}$ into the Kronecker sum of three smaller matrices, each of which is transformed into a multiple of binary matrix with one's only on the superdiagonal and subdiagonal, and zeros elsewhere. After revealing the canonical Jordan forms and the sets of eigenvalues of these matrices, we derive an explicit formula of the generalized Jordan forms of $T_{RN}$ and $T_{RN}^{-1}$, if it exists. Before demonstrating Theorem \ref{thm:main-theorem-algorithm-for-JT-and-Inverse}, we use the following two examples to elaborate the idea of the proof.

\begin{example}\label{eg:m=n=s=2-p-not3}
Suppose $n=s=m=2$ and $k_n \equiv k_s \equiv k_m \pmod{p}$. For the sake of simplicity, we assume that $a=b=c=d=e=f=1$; in this case, $k_n \equiv k_s \equiv k_m \equiv 1 \pmod{p}$.

Let
$$
U_2 \equiv \begin{pmatrix}
1 & 1 \\
1 & -1
\end{pmatrix} \pmod{p};
$$
it is seen that $U_2^{-1} K_2 U_2 \equiv \mathrm{diag} (1, -1) \pmod{p}$. Next, consider
$$ 
U_4 \equiv \begin{pmatrix}
1 & 1 & 1 & 1 \\
1 & 1 & -1 & -1 \\
1 & -1 & 1 & -1 \\
1 & -1 & 1 & 1
\end{pmatrix} \pmod{p}
$$
and
$$
U_8 \equiv \begin{pmatrix}
1 & 1 & 1 & 1 & 1 & 1 & 1 & 1 \\
1 & 1 & 1 & 1 & -1 & -1 & -1 & -1 \\
1 & 1 & -1 & -1 & 1 & 1 & -1 & -1 \\
1 & 1 & -1 & -1 & 1 & 1 & 1 & 1 \\
1 & -1 & 1 & -1 & 1 & -1 & 1 & -1 \\
1 & -1 & 1 & -1 & -1 & 1 & -1 & 1 \\
1 & -1 & -1 & 1 & 1 & -1 & -1 & 1 \\
1 & -1 & -1 & 1 & 1 & -1 & 1 & -1
\end{pmatrix} \pmod{p};
$$
it is easily verified that $U_4^{-1} M_4 U_4 \equiv \mathrm{diag} (2, 0, 0, -2) \pmod{p}$ and $U_8^{-1} T_{RN} U_8 \equiv \mathrm{diag} (3, 1, 1, -1, 1, -1, -1, -3) \pmod{p}$, which is reversible if and only if $p\neq 3$ (cf.~Corollary \ref{cor:m=n=s=2-invertible-condition} (i)). A straightforward examination shows that
$$
U_8^{-1} \equiv 4^{-1} \begin{pmatrix}
0 & 1 & 0 & 1 & 0 & 1 & 0 & 1 \\
0 & 1 & 0 & 1 & 0 & -1 & 0 & -1 \\
1 & 0 & 0 & -1 & 1 & 0 & 0 & -1 \\
1 & 0 & 0 & -1 & -1 & 0 & 0 & 1 \\
1 & -1 & 1 & -1 & 1 & -1 & 1 & -1 \\
1 & -1 & 1 & -1 & -1 & 1 & -1 & 1 \\
0 & 0 & -1 & 1 & 0 & 0 & -1 & 1 \\
0 & 0 & -1 & 1 & 0 & 0 & 1 & -1
\end{pmatrix} \pmod{p}.
$$
Hence, we can conclude that
\begin{align*}
T_{RN}^{-1} &\equiv U_8 \cdot \mathrm{diag} (3^{-1}, 1, 1, -1, 1, -1, -1, -3^{-1}) \cdot U_8^{-1} \pmod{p} \\
 & \equiv 3^{-1} \begin{pmatrix}
0 & 1 & 1 & 0 & 3 & -2 & 2 & -4 \\
0 & 1 & -1 & 2 & 0 & 1 & -2 & 0 \\
0 & 1 & -1 & 2 & 0 & -2 & 1 & 0 \\
0 & 1 & 1 & 0 & 0 & -2 & 2 & -1 \\
3 & -2 & 2 & -4 & 0 & 1 & 1 & 0 \\
0 & 1 & -2 & 0 & 0 & 1 & -1 & 2 \\
0 & -2 & 1 & 0 & 0 & 1 & -1 & 2 \\
0 & -2 & 2 & -1 & 0 & 1 & 1 & 0
\end{pmatrix} \pmod{p}.
\end{align*}
\end{example}

\begin{example}\label{eg:m=n=s=4-p=5-semiJordan}
Suppose that $n = s = m = 4$ and $p = 5$. It follows that
$$
g_4(x) = x^4 - 3x + 1 = (x - 2)^2 (x - 3)^2 \pmod{5}
$$
splits in $\mathbb{Z}_5 [x]$. Write $\mathbb{Z}_5^* = \{1, 4\} \bigcup \{2, 3\}$. A straightforward examination demonstrates that $k_{a, b}, k_{c, d}, k_{e, f} \in \{1, 2\}$ if $\{a, b\}, \{c, d\}$, and $\{e, f\}$ are in the same partition, respectively.

Since $\mathbf{R}_4 = \{2, 3\}$, Theorem \ref{thm:eigenvalue-set-decomposition} elaborates that the set of eigenvalues of $T_{RN}$ is
\begin{align*}
\mathbf{E}_T &= k_{a, b} \mathbf{R}_4 + k_{c, d} \mathbf{R}_4 + k_{e, f} \mathbf{R}_4 \\
&= \{t_1 k_{a, b} + t_2 k_{c, d} + t_3 k_{e, f}: t_1, t_2, t_3 = 2, 3\}.
\end{align*}
A careful examination indicates that $T_{RN}$ is reversible if and only if the triple $(k_{a, b}, k_{c, d}, k_{e, f})$ satisfies one of the following:
$$
(1, 1, 1), (1, 1, 4), (1, 4, 4), (2, 2, 2), (2, 2, 3), (2, 3, 3), (3, 3, 3), (4, 4, 4).
$$

Let
$$
P_4(t_1, t_2) = \begin{pmatrix}
k_{t_1, t_2^{-1}} & 0 & 0 & 0 \\
0 & k_{t_1, t_2^{-1}}^2 & 0 & 0 \\
0 & 0 & k_{t_1, t_2^{-1}}^3 & 0 \\
0 & 0 & 0 & k_{t_1, t_2^{-1}}^4
\end{pmatrix}
\quad \text{and} \quad
U = \begin{pmatrix}
1 & 0 & 1 & 0 \\
2 &1 & 3 & 1 \\
3 & 4 & 3 & 1 \\
4 & 0 & 1 & 0
\end{pmatrix}
$$
provided $k_{t_1, t_2^{-1}}$ exists. Then
$$
U^{-1} P_4(d, c)^{-1} S_4(c, d) P_4(d, c) U = k_{c, d} \begin{pmatrix}
2 & 1 & 0 & 0 \\
0 & 2 & 0 & 0 \\
0 & 0 & 3 & 1 \\
0 & 0 & 0 & 3
\end{pmatrix} =: k_{c, d} J_4.
$$
Furthermore,
$$
P_{16}(a, b)^{-1} M_{16} P_{16}(a, b) = I_4 \otimes S_4(c, d) + (k_{a, b} K_4) \otimes I_4,
$$
where $P_{16}(a, b) = P_{4}(a, b) \otimes I_4$. Notably,
$$
(I_4 \otimes (U^{-1} P_4(d, c)^{-1})) \cdot (I_4 \otimes S_4(c, d)) \cdot (I_4 \otimes P_4(d, c) U) = I_4 \otimes (k_{c, d} J_4),
$$
and
$$
(U^{-1} \otimes I_4) \cdot ((k_{a, b} K_4) \otimes I_4) \cdot (U \otimes I_4) = (k_{a, b} J_4) \otimes I_4.
$$
Let $\widetilde{U} = U \otimes (P_4(d, c) U)$. It follows that
\begin{align*}
&\widetilde{U}^{-1} P_{16}(a, b)^{-1} M_{16} P_{16}(a, b) \widetilde{U} = I_4 \otimes (k_{c, d} J_4) + (k_{a, b} J_4) \otimes I_4 \\
&= \begin{pmatrix}
k_{c, d} J_4 + 2 k_{a, b} I_4 & k_{a, b} I_4 & O_4 & O_4 \\
O_4 & k_{c, d} J_4 + 2 k_{a, b} I_4 & O_4 & O_4 \\
O_4 & O_4 & k_{c, d} J_4 + 3 k_{a, b} I_4 & k_{a, b} I_4 \\
O_4 & O_4 & O_4 & k_{c, d} J_4 + 3 k_{a, b} I_4
\end{pmatrix} \\
& =: J_{16}.
\end{align*}

Set $P_{64}(e, f) = P_4(e, f) \otimes I_{16}$, then
$$
P_{64}(e, f)^{-1} T_{RN} P_{64}(e, f) = I_4 \otimes M_{16} + (k_{e, f} K_4) \otimes I_{16}.
$$
Let $\widehat{U} = U \otimes \widetilde{U}$, it is seen that
\begin{align*}
&\widehat{U}^{-1} P_{64}(e, f)^{-1} T_{RN} P_{64}(e, f) \widehat{U} = I_4 \otimes J_{16} + (k_{e, f} J_4) \otimes I_{16} \\
&= \begin{pmatrix}
J_{16} + 2 k_{e, f} I_{16} & k_{e, f} I_{16} & O_{16} & O_{16} \\
O_{16} & J_{16} + 2 k_{e, f} I_{16} & O_{16} & O_{16} \\
O_{16} & O_{16} & J_{16} + 3 k_{e, f} I_{16} & k_{e, f} I_{16} \\
O_{16} & O_{16} & O_{16} & J_{16} + 3 k_{e, f} I_{16} 
\end{pmatrix} \\
& =: J_{T_{RN}} = \begin{pmatrix}
A_1 & k_{e, f} I_{16} & O_{16} & O_{16} \\
O_{16} & A_1 & O_{16} & O_{16} \\
O_{16} & O_{16} & A_2 & k_{e, f} I_{16} \\
O_{16} & O_{16} & O_{16} & A_2
\end{pmatrix}.
\end{align*}
Write
\begin{align*}
A_1 = \begin{pmatrix}
B_{1, 1} & k_{a, b} I_4 & O_4 & O_4 \\
O_4 & B_{1, 1} & O_4 & O_4 \\
O_4 & O_4 & B_{1, 2} & k_{a, b} I_4 \\
O_4 & O_4 & O_4 & B_{1, 2}
\end{pmatrix},
A_2 = \begin{pmatrix}
B_{2, 1} & k_{a, b} I_4 & O_4 & O_4 \\
O_4 & B_{2, 1} & O_4 & O_4 \\
O_4 & O_4 & B_{2, 2} & k_{a, b} I_4 \\
O_4 & O_4 & O_4 & B_{2, 2}
\end{pmatrix},
\end{align*}
where
\begin{align*}
B_{1, 1} &= k_{c, d} J_4 + (2 k_{a, b} + 2 k_{e, f}) I_4, & B_{1, 2} &= k_{c, d} J_4 + (3 k_{a, b} + 2 k_{e, f}) I_4, \\
B_{2, 1} &= k_{c, d} J_4 + (2 k_{a, b} + 3 k_{e, f}) I_4, & B_{2, 2} &= k_{c, d} J_4 + (3 k_{a, b} + 3 k_{e, f}) I_4.
\end{align*}
To increase the readability, we assume that $k_{c, d}=k_{a, b}=1, k_{e, f}=4$. It is easily seen that
\begin{align*}
B_{1, 1}^{-1} &= \begin{pmatrix}
3 & 1 & 0 & 0 \\
0 & 3 & 0 & 0 \\
0 & 0 & 2 & 1 \\
0 & 0 & 0 & 2
\end{pmatrix}, &
B_{1, 2}^{-1} &= \begin{pmatrix}
2 & 1 & 0 & 0 \\
0 & 2 & 0 & 0 \\
0 & 0 & 4 & 4 \\
0 & 0 & 0 & 4
\end{pmatrix}, \\
B_{2, 1}^{-1} &= \begin{pmatrix}
1 & 4 & 0 & 0 \\
0 & 1 & 0 & 0 \\
0 & 0 & 3 & 1 \\
0 & 0 & 0 & 3
\end{pmatrix}, &
B_{2, 2}^{-1} &= \begin{pmatrix}
3 & 1 & 0 & 0 \\
0 & 3 & 0 & 0 \\
0 & 0 & 2 & 1 \\
0 & 0 & 0 & 2
\end{pmatrix}.
\end{align*}
Lemma \ref{lem:Inverse-of-Jordan-form} illustrates that
\begin{align*}
A_1^{-1} = \begin{pmatrix}
B_{1, 1}^{-1} & B_{1,3} & O_4 & O_4 \\
O_4 & B_{1, 1}^{-1} & O_4 & O_4 \\
O_4 & O_4 & B_{1, 2}^{-1} & B_{1,4} \\
O_4 & O_4 & O_4 & B_{1, 2}^{-1}
\end{pmatrix},
A_2^{-1} = \begin{pmatrix}
B_{2, 1}^{-1} & B_{2,3} & O_4 & O_4 \\
O_4 & B_{2, 1}^{-1} & O_4 & O_4 \\
O_4 & O_4 & B_{2, 2}^{-1} & B_{2,4} \\
O_4 & O_4 & O_4 & B_{2, 2}^{-1}
\end{pmatrix},
\end{align*}
where
\begin{align*}
B_{1,3} &= \begin{pmatrix}
1 & 4 & 0 & 0 \\
0 & 1 & 0 & 0 \\
0 & 0 & 1 & 1 \\
0 & 0 & 0 & 1
\end{pmatrix}, &
B_{1,4} &= \begin{pmatrix}
1 & 1 & 0 & 0 \\
0 & 1 & 0 & 0 \\
0 & 0 & 4 & 3 \\
0 & 0 & 0 & 4
\end{pmatrix}, \\
B_{2,3} &= \begin{pmatrix}
4 & 2 & 0 & 0 \\
0 & 4 & 0 & 0 \\
0 & 0 & 1 & 1 \\
0 & 0 & 0 & 1
\end{pmatrix}, &
B_{2,4} &= \begin{pmatrix}
1 & 1 & 0 & 0 \\
0 & 1 & 0 & 0 \\
0 & 0 & 1 & 1 \\
0 & 0 & 0 & 1
\end{pmatrix}.
\end{align*}
Analogous calculation to the above demonstrates that
$$
T_{RN}^{-1} = P_{64}(e, f) \widehat{U} J_{T_{RN}}^{-1} \widehat{U}^{-1} P_{64}(e, f)^{-1},
$$
where
$$
J_{T_{RN}}^{-1} = \begin{pmatrix}
A_1^{-1} & A_3 & O_{16} & O_{16} \\
O_{16} & A_1^{-1} & O_{16} & O_{16} \\
O_{16} & O_{16} & A_2^{-1} & A_4 \\
O_{16} & O_{16} & O_{16} & A_2^{-1}
\end{pmatrix}
$$
with
$$
A_3 = 4 \begin{pmatrix}
B_{1,3} & A_{3,1} & O_4 & O_4 \\
O_4 & B_{1,3} & O_4 & O_4 \\
O_4 & O_4 & B_{1,4} & A_{3,2} \\
O_4 & O_4 & O_4 & B_{1,4}
\end{pmatrix},
A_4 = 4 \begin{pmatrix}
B_{2,3} & A_{4,1} & O_4 & O_4 \\
O_4 & B_{2,3} & O_4 & O_4 \\
O_4 & O_4 & B_{2,4} & A_{4,2} \\
O_4 & O_4 & O_4 & B_{2,4}
\end{pmatrix},
$$
and
\begin{align*}
A_{3,1} &= \begin{pmatrix}
4 & 4 & 0 & 0 \\
0 & 4 & 0 & 0 \\
0 & 0 & 1 & 4 \\
0 & 0 & 0 & 1
\end{pmatrix}, &
A_{3,2} &= \begin{pmatrix}
1 & 4 & 0 & 0 \\
0 & 1 & 0 & 0 \\
0 & 0 & 3 & 4 \\
0 & 0 & 0 & 3
\end{pmatrix}, \\
A_{4,1} &= \begin{pmatrix}
2 & 4 & 0 & 0 \\
0 & 2 & 0 & 0 \\
0 & 0 & 4 & 2 \\
0 & 0 & 0 & 4
\end{pmatrix}, &
A_{4,2} &= \begin{pmatrix}
4 & 2 & 0 & 0 \\
0 & 4 & 0 & 0 \\
0 & 0 & 1 & 4 \\
0 & 0 & 0 & 1
\end{pmatrix}.
\end{align*}
\end{example}

\begin{proof}[Proof of Theorem \ref{thm:main-theorem-algorithm-for-JT-and-Inverse}]
Since $(a, b), (c, d), (e, f)$ satisfy the quadratic reciprocity law, it is easily seen that $k_{a,b^{-1}}, k_{c, d^{-1}}$, and $k_{e^{-1}, f}$ are well-defined. Let $\mathbb{E}$ be the splitting field for $g_n^{[p]}(x), g_s^{[p]}(x)$, and $g_m^{[p]}(x)$. For $j \in \{n, s, m\}$, there exists $U_j \in \mathcal{M}_j (\mathbb{E})$ such that $U_j^{-1} K_j U_j \equiv J_j$ is the canonical Jordan form of $K_j$ in $\mathbb{E}$. We divide the proof into several steps.

\noindent \textbf{Step 1.} Let $P_{c, d} = \mathrm{diag}(k_{c, d^{-1}}, k_{c, d^{-1}}^2, \cdots, k_{c, d^{-1}}^n)$ and let $U_{S_n} = P_{c, d} U_n$. It follows that $U_{S_n}^{-1} S_n(c, d) U_{S_n} = k_{c, d} J_n$, where
$$
J_n = \begin{pmatrix}
\lambda_{n, 1} & \epsilon_{n, 1} & 0 & \cdots & 0 \\
0 & \lambda_{n, 2} & \epsilon_{n, 2} & \cdots & 0 \\
\vdots & \vdots & \ddots & \vdots & \vdots \\
0 & 0 & \cdots & \lambda_{n, n-1} & \epsilon_{n, n-1} \\
0 & 0 & \cdots & 0 & \lambda_{n, n}
\end{pmatrix},
$$
and $\epsilon_{n, \ell} \in \{0, 1\}$ for $1 \leq \ell \leq n-1$.

\noindent \textbf{Step 2.} Let $P_{a,b} = \mathrm{diag}(k_{a,b^{-1}}, k_{a,b^{-1}}^2, \cdots, k_{a,b^{-1}}^s) \otimes I_{n}$. Since $M_s = I_s \otimes S_n(c, d) + S_s(a, b) \otimes I_n$, we can derive that
$$
P_{a, b}^{-1} M_s P_{a, b} =  I_s \otimes S_n(c, d) + (k_{a, b} K_s) \otimes I_n.
$$
Let $U_{M_s} = P_{a, b} \cdot (U_s \otimes U_{S_n})$. Then
\begin{align*}
U_{M_s}^{-1} M_s U_{M_s} &= (U_s^{-1} \otimes U_{S_n}^{-1}) (P_{a, b}^{-1} M_s P_{a, b}) (U_s \otimes U_{S_n}) \\
&= (U_s^{-1} \otimes U_{S_n}^{-1}) (I_s \otimes S_n(c, d) + (k_{a, b} K_s) \otimes I_n) (U_s \otimes U_{S_n}) \\
&= I_s \otimes (k_{c, d} J_n) + (k_{a, b} J_s) \otimes I_n =: J_{M_s}.
\end{align*}

\noindent \textbf{Step 3.} Let $P_{e,f} = \mathrm{diag}(k_{e^{-1}, f}, k_{e^{-1}, f}^2, \cdots, k_{e^{-1}, f}^m) \otimes I_{ns}$ and let $U_{T_{RN}} = P_{e,f} \cdot (U_m \otimes U_{M_s})$. It follows that
$$
P_{e,f}^{-1} T_{RN} P_{e,f} = I_m \otimes M_s + (k_{e,f} K_m) \otimes I_{ns}
$$
and
\begin{align*}
U_{T_{RN}}^{-1} T_{RN} U_{T_{RN}} &= I_m \otimes J_{M_s} + (k_{e,f} J_m) \otimes I_{ns} \\
&= \begin{pmatrix}
A_1 & \epsilon_{m,1} I_{ns} & O_{ns} & \cdots & O_{ns} \\
O_{ns} & A_2 & \epsilon_{m,2} I_{ns} & \cdots & O_{ns} \\
\vdots & \vdots & \ddots & \vdots & \vdots \\
O_{ns} & O_{ns} & \cdots & A_{m-1} & \epsilon_{m,m-1} I_{ns} \\
O_{ns} & O_{ns} & \cdots & O_{ns} & A_m
\end{pmatrix} \\
&=: J_{T_{RN}},
\end{align*}
where $A_i = J_{M_s} + k_{e,f} \lambda_{m,i} I_{ns}$ for $1 \leq i \leq m$, and $\epsilon_{m, i} \in \{0, 1\}$ for $1 \leq i \leq m-1$. The desired generalized Jordan form $J_{T_{RN}}$ of $T_{RN}$ is then obtained.

\noindent \textbf{Step 4.} Suppose that $T_{RN}$ is reversible. Lemma \ref{lem:Inverse-of-Jordan-form} asserts that the explicit expression of $T_{RN}^{-1} = U_{T_{RN}} \cdot J_{T_{RN}}^{-1} \cdot U_{T_{RN}}$ follows immediately from the calculation of $A_i^{-1}$ for $1 \leq i \leq m$. Notably,
$$
A_i = \begin{pmatrix}
B_{i, 1} & \epsilon_{s, 1} I_n & O_n & \cdots & O_n \\
O_n & B_{i, 2} & \epsilon_{s, 2} I_n & \cdots & O_n \\
\vdots & \vdots & \ddots & \vdots & \vdots \\
O_n & O_n & \cdots & B_{i, s-1} & \epsilon_{s, s-1} I_n \\
O_n & O_n & \cdots & O_n & B_{i, s}
\end{pmatrix},
$$
where $B_{i, j} = k_{c, d} J_n + (k_{a, b} \lambda_{s, j} + k_{e,f} \lambda_{m,i}) I_n$ for $1 \leq j \leq s$. Lemma \ref{lem:Inverse-of-Jordan-form} demonstrates that
$$
B_{i, j}^{-1} = \begin{pmatrix}
w_{i,j,1}^{-1} & - \epsilon_{n,1} w_{i,j,1}^{-1} w_{i,j,2}^{-1} & \cdots & (-1)^{n-1} \prod\limits_{\ell = 1}^{n-1} \epsilon_{n,\ell} \prod\limits_{\ell=1}^n w_{i,j,\ell}^{-1} \\
0 & w_{i,j,2}^{-1} & \cdots & (-1)^{n-2} \prod\limits_{\ell = 2}^{n-1} \epsilon_{n,\ell} \prod\limits_{\ell=2}^n w_{i,j,\ell}^{-1} \\
\vdots & \vdots & \ddots & \vdots \\
0 & \cdots & 0 & w_{i,j,n}^{-1}
\end{pmatrix},
$$
where $w_{i,j,\ell} = k_{c, d} \lambda_{n, \ell} + k_{a,b} \lambda_{s, j} + k_{e,f} \lambda_{m, i}$ for $1 \leq \ell \leq n$.

\noindent \textbf{Step 5.} The desired algorithm for deriving the generalized Jordan form of $T_{RN}$ and its inverse matrix, if it exists, is as follows.
\begin{enumerate}[\sc {JFA}1.]
\item Find $U_r$ such that $U_r^{-1} K_r U_r \equiv J_r$ is a canonical Jordan form over the splitting field for $\{g_n^{[p]}, g_s^{[p]}, g_m^{[p]}\}$, where $r = n, s, m$.

\item Let
\begin{align*}
P_{c, d} &= \mathrm{diag}(k_{c, d^{-1}}, k_{c, d^{-1}}^2, \cdots, k_{c, d^{-1}}^n), & U_{S_n} &= P_{c, d} \cdot U_n, \\
P_{a,b} &= \mathrm{diag}(k_{a,b^{-1}}, k_{a,b^{-1}}^2, \cdots, k_{a,b^{-1}}^s) \otimes I_{n}, & U_{M_s} &= P_{a, b} \cdot (U_s \otimes U_{S_n}), \\
P_{e,f} &= \mathrm{diag}(k_{e^{-1}, f}, k_{e^{-1}, f}^2, \cdots, k_{e^{-1}, f}^m) \otimes I_{ns}, & U_{T_{RN}} &= P_{e,f} \cdot (U_m \otimes U_{M_s}).
\end{align*}
Then $U_{T_{RN}}^{-1} T_{RN} U_{T_{RN}}$ is the desired generalized Jordan form of $T_{RN}$.

\item Let $\{\epsilon_{r_1, r_2}\}_{r_2=1}^{r_1-1} \subseteq \{0, 1\}$ be the set obtained from $J_{r_1}$, where $r_1 = m, n, s$. Define
$$
w_{i,j,\ell} = k_{c, d} \lambda_{n, \ell} + k_{a,b} \lambda_{s, j} + k_{e,f} \lambda_{m, i},
$$
where $1 \leq i \leq m, 1 \leq j \leq s$, and $1 \leq \ell \leq n$. Furthermore, let
$$
C_{i,j}(q_1, q_2) = \left\{\begin{aligned}
&w_{i,j,q_1}^{-1}, & & q_1=q_2; \\
&(-1)^{q_2-q_1} \prod\limits_{r = q_1}^{q_2-1} \epsilon_{n,r} \prod\limits_{r=q_1}^{q_2} w_{i,j,r}^{-1}, & & q_1<q_2; \\
&0, & & q_1 > q_2;
\end{aligned}\right.
$$
for $1 \leq q_1, q_2 \leq n$,
$$
D_i(q_1, q_2) = \left\{\begin{aligned}
&C_{i,q_1}, & & q_1=q_2; \\
&(-1)^{q_2-q_1} \prod\limits_{r = q_1}^{q_2-1} \epsilon_{s,r} \prod\limits_{r=q_1}^{q_2} C_{i,r}^{-1}, & & q_1<q_2; \\
&O_{n}, & & q_1 > q_2;
\end{aligned}\right.
$$
for $1 \leq q_1, q_2 \leq s$. We obtain
$$
J_{T_{RN}}^{-1}(q_1, q_2) = \left\{\begin{aligned}
&D_{q_1}^{-1}, & & q_1=q_2; \\
&(-1)^{q_2-q_1} \prod\limits_{r = q_1}^{q_2-1} \epsilon_{m,r} \prod\limits_{r=q_1}^{q_2} D_{r}^{-1}, & & q_1<q_2; \\
&O_{ns}, & & q_1 > q_2;
\end{aligned}\right.
$$
herein $1 \leq q_1, q_2 \leq m$. The inverse matrix of $T_{RN}$ then follows immediately.
\end{enumerate}

This completes the proof.
\end{proof}

\begin{remark}
It can be verified without difficulty that $T_{RN}$ is diagonalizable if and only if $K_n, K_s$, and $K_m$ are all diagonalizable. Furthermore, $J_{T_{RN}}$ is a canonical Jordan form if and only if $K_s$ and $K_m$ are both diagonalizable.
\end{remark}

\section{Three Dimensional Cellular Automata: General Cases} \label{sec:3d-ca-general}

The study of the reversibility, generalized Jordan form, and the inverse matrix, if it exists, of $T_{RN}$ can extend to more general cases. This section is devoted to the discussion of general conditions. We start with the case where $c_0 = 0$.

Recall that $T_{RN} = I_m \otimes M_s + S_m(f, e) \otimes I_{ns}$ with $M_s = I_s \otimes S_n(c, d) + S_s(a, b) \otimes I_n$; it is essential to characterize the property of the matrix
$$
S_k (t_1, t_2) = \begin{pmatrix}
0 & t_2 & 0 & 0 & \cdots & 0 \\
t_1 & 0 & t_2 & 0 & \cdots & 0 \\
0 & t_1 & 0 & t_2 & \cdots & 0 \\
\vdots & \vdots & \ddots & \ddots & \ddots & \vdots \\
0 & 0 & \cdots & t_1 & 0 & t_2 \\
0 & 0 & \cdots & 0 & t_1 & 0
\end{pmatrix}_{k \times k}.
$$

For $t_1, t_2 \in \mathbb{Z}_p$ and $j \in \mathbb{N}$, define
\begin{equation}\label{eq:charpoly-for-Sj-t1-t2}
g_{j; t_1, t_2} (x) = \sum\limits_{i=0}^{[j/2]} (-1)^i (t_1 t_2)^i {{j-i}\choose{i}} x^{j-2i}.
\end{equation}
Let $\mathbb{E}$ denote the splitting field for $g_{n; c, d}^{[p]} (x), g_{s; a, b}^{[p]} (x)$, and $g_{m; f, e}^{[p]} (x)$, and let $\mathbf{R}_{j; t_1, t_2}$ be the collection of roots of $g_{j; t_1, t_2}^{[p]} (x)$ in $\mathbb{E}$. Similar to Theorem \ref{thm:eigenvalue-set-decomposition}, the reversibility of $T_{RN}$ is revealed after we characterize its eigenvalues.

\begin{theorem}\label{thm:general-eigenvalue-set-decomposition}
The set $\mathbf{E}_{T_{RN}}$ of eigenvalues of $T_{RN}$ is
$$
\mathbf{E}_{T_{RN}} = \mathbf{R}_{n; c, d} + \mathbf{R}_{s; a, b} +\mathbf{R}_{m; e, f}.
$$
\end{theorem}
\begin{proof}
Similar to the proof of Theorem \ref{thm:eigenvalue-set-decomposition}, it suffices to show that $g_{j; t_1, t_2}(x)$ is the characteristic polynomial of $S_j(t_1, t_2)$ since the set of eigenvalues of $T_{RN}$ is
$$
\mathbf{E}_{T_{RN}} = \mathbf{E}_{S_n(c,d)} + \mathbf{E}_{S_s(a,b)} +\mathbf{E}_{S_m(f,e)}.
$$
Let $g_{j; t_1, t_2}(x) = \det (x I_j - S_j(t_1, t_2))$ be the characteristic polynomial of $S_j (t_1, t_2)$. Set $g_{0; t_1, t_2}(x) = 1$ and $g_{j; t_1, t_2}(x) = 0$ for $j < 0$. It is easily seen that
$$
g_{j; t_1, t_2}(x) = x g_{j-1; t_1, t_2}(x) - t_1 t_2 g_{j-2; t_1, t_2}(x), \qquad j \geq 1.
$$
Let $G(u, x) = \sum\limits_{j \geq 0} g_{j; t_1, t_2}(x) u^j$ be the generating function. Then
$$
G(u, x) = \dfrac{1}{t_1 t_2 u^2 - xu + 1} = \sum_{j \geq 0} (u (x - t_1 t_2 u))^j.
$$
It follows immediately that
$$
g_{j; t_1, t_2} = \sum\limits_{i=0}^{[j/2]} (-1)^i (t_1 t_2)^i {{j-i}\choose{i}} x^{j-2i}.
$$
The proof is complete.
\end{proof}

\begin{proposition}\label{prop:roots-of-gj-t1t2}
Suppose that $k < \ell$ and $t_1 t_2 = q_1 q_2$. Let $h(x) = \mathrm{gcd}(g_{k; t_1, t_2}(x), g_{\ell; q_1, q_2}(x))$. Then
\begin{enumerate}[\bf (i)]
\item $\mathrm{deg}~h(x) = \mathrm{gcd}(k+1, \ell + 1) - 1$;
\item $h(x) = g_{k; t_1, t_2}(x)$ if and only if $(k+1) | (\ell + 1)$.
\end{enumerate}
\end{proposition}
\begin{proof}
The proof is similar to the proof of Proposition \ref{prop:roots-of-gj}, thus it is omitted.
\end{proof}

\begin{corollary}
$T_{RN}$ is reversible if and only if
$$
0 \notin \mathbf{R}_{n; c, d} + \mathbf{R}_{s; a, b} +\mathbf{R}_{m; e, f}.
$$
\end{corollary}

\begin{theorem}\label{thm:general-main-theorem-algorithm-for-JT-and-Inverse}
There is an algorithm for the computation of the generalized Jordan form of $T_{RN}$ and $T_{RN}^{-1}$, if it exists.
\end{theorem}
\begin{proof}
The proof is similar to the discussion in the proof of Theorem \ref{thm:main-theorem-algorithm-for-JT-and-Inverse}, thus we only sketch the outline.

Given $t_1, t_2 \in \mathbb{Z}_p$ and $k \in \mathbb{N}$, let $U_{k} (t_1, t_2) \in \mathcal{M}_k(\mathbb{E})$ be the matrix consists of the generalized eigenvectors of $S_k(t_1, t_2)$; in other words,
\begin{align*}
&U_{k}^{-1} (t_1, t_2) S_k(t_1, t_2) U_{k} (t_1, t_2) \\
&= \begin{pmatrix}
\lambda_{k, 1} (t_1, t_2) & \epsilon_{k, 1}(t_1, t_2) & 0 & \cdots & 0 \\
0 & \lambda_{k, 2} (t_1, t_2) & \epsilon_{k, 2}(t_1, t_2) & \cdots & 0 \\
\vdots & \vdots & \ddots & \ddots & \vdots \\
0 & 0 & \cdots & \lambda_{k, k-1} (t_1, t_2) & \epsilon_{k, k-1}(t_1, t_2) \\
0 & 0 & \cdots & 0 & \lambda_{k, k} (t_1, t_2)
\end{pmatrix} \\ &=: J_k (t_1, t_2),
\end{align*}
where $\epsilon_{k, r}(t_1, t_2) \in \{0, 1\}$ for $1 \leq r \leq k$.

Let $U_{M_s} = U_s(a, b) \otimes U_n(c, d)$. It follows that
$$
U_{M_s}^{-1} M_s U_{M_s} = I_s \otimes J_n(c, d) + J_s(a, b) \otimes I_n =: J_{M_s}
$$
is a generalized Jordan form of $M_s$ over $\mathbb{E}$. Furthermore, let $U_{T_{RN}} = U_m(f, e) \otimes U_{M_s}$. Then
$$
U_{T_{RN}}^{-1} T_{RN} U_{T_{RN}} = I_m \otimes J_{M_s} + J_m(f, e) \otimes I_{ns} =: J_{T_{RN}}
$$
is the desired generalized Jordan form of $T_{RN}$.

Suppose that $T_{RN}$ is reversible. Let
\begin{align*}
A_i &= J_{M_s} + \lambda_{m,i}(f,e) I_{ns}, \qquad 1 \leq i \leq m, \\
B_{i, j} &= J_n(c, d) + (\lambda_{s,j}(a,b) + \lambda_{m,i}(f,e)) I_n, \qquad 1 \leq j \leq s.
\end{align*}
Then the diagonal and the superdiagonal of $J_{T_{RN}}$ are $\{A_i\}_{i=1}^m$ and $\{\epsilon_{m,r}(f,e) I_{ns}\}_{r=1}^{m-1}$, respectively, and the diagonal and the superdiagonal of $A_i$ are $\{B_{i,j}\}_{j=1}^s$ and $\{\epsilon_{s,q}(a,b) I_{ns}\}_{q=1}^{s-1}$, respectively. Repeatedly applying Lemma \ref{lem:Inverse-of-Jordan-form} reveals the formulae of $B_{i,j}^{-1}, A_i^{-1}$, and $J_{T_{RN}}^{-1}$, respectively; this completes the proof.
\end{proof}

\begin{remark}
It can be verified without difficulty that $T_{RN}$ is diagonalizable if and only if $S_n(c, d), S_s(a, b)$, and $S_m(f,e)$ are all diagonalizable, and $J_{T_{RN}}$ is a canonical Jordan form if and only if $S_s(a, b)$ and $S_m(f,e)$ are both diagonalizable. Indeed, it is seen from the proof of Theorem \ref{thm:general-main-theorem-algorithm-for-JT-and-Inverse} that $T_{RN}$ is diagonalizable if and only if $J_{M_s}$ and $J_m$ are both diagonal. Furthermore, $J_{M_s}$ is diagonal if and only if both $J_n$ and $J_s$ are both diagonal. Therefore, we conclude that $T_{RN}$ is diagonalizable if and only if $S_n(c, d), S_s(a, b)$, and $S_m(f,e)$ are all diagonalizable. The other statement can be derived analogously, thus it is omitted.
\end{remark}

\begin{remark}
In the case where $c_0 \neq 0$, we substitute $S_n(c,d)$ as $S' = S_n(c,d) + c_0 I_n$, then Theorems \ref{thm:general-eigenvalue-set-decomposition} still works provided that $\mathbf{R}_{n; c, d}$ is replaced by $\mathbf{R}_n'$, the collection of roots of $g_{n; c, d}^{[p]}(x-c_0)$. Furthermore, the algorithm for the computation of the generalized Jordan form of $T_{RN}$ (Theorem \ref{thm:general-main-theorem-algorithm-for-JT-and-Inverse}) remains to be true with a minor modification.
\end{remark}

\section{Reversibility for Multidimensional Cellular Automata} \label{sec:multidimension-ca}

This section extends the results in Sections \ref{sec:3d-ca-square-root} and \ref{sec:3d-ca-general} to multidimensional linear cellular automata with the prolonged $\eta$-nearest neighborhood for $\eta \in \mathbb{N}$. The demonstration is analogous to the discussion in the previous sections, thus it is omitted.

\subsection{Nearest Neighborhood} \label{subsec:multidimension-nesrest-nbd}

Let $n \in \mathbb{N}$, $n \geq 2$, and let $\mathbb{Z}_p^{\mathbb{Z}^n}$ be the $n$-dimensional lattice over finite field $\mathbb{Z}_p$. Suppose that $\{e_k\}_{k=1}^n$ is the standard basis of $\mathbb{R}^n$; set
$$
\mathcal{N} = \{v \in \mathbb{Z}^n: v = \lambda e_k \text{ for some } k \in \{1, \ldots, n\} \text{ and } \lambda \in \{-1, 0, 1\}\}.
$$
Fix $c, \ell_k, r_k \in \mathbb{Z}_p$ for $1 \leq k \leq n$; define $\phi: \mathbb{Z}_p^{\mathcal{N}} \to \mathbb{Z}_p$ as
$$
\phi(y_{\mathcal{N}}) = c y_{\mathbf{0}} + \sum_{k=1}^n (\ell_k y_{-e_k} + r_k y_{e_k}) \pmod{p}
$$
An $n$-dimensional linear cellular automaton $\Phi: \mathbb{Z}_p^{\mathbb{Z}^n} \to \mathbb{Z}_p^{\mathbb{Z}^n}$ with nearest neighborhood is defined as
$$
\Phi(X)_{\mathbf{i}} = \phi(X_{\mathbf{i} + \mathcal{N}}) = c X_{\mathbf{i}} + \sum_{k=1}^n (\ell_k X_{\mathbf{i} - e_k} + r_k X_{\mathbf{i} + e_k}) \pmod{p}
$$
for every $\mathbf{i} \in \mathbb{Z}^n$. Given $m_1, m_2, \ldots, m_n \in \mathbb{N}$, $m_k \geq 2$ for $1 \leq k \leq n$, a linear cellular automaton under null boundary condition is described as
$$
\Phi_N(X)_{\mathbf{i}} = c X_{\mathbf{i}} + \sum_{k=1}^n (\ell_k(\mathbf{i}) X_{\mathbf{i} - e_k} + r_k(\mathbf{i}) X_{\mathbf{i} + e_k}) \pmod{p}
$$
where
$$
\ell_k(\mathbf{i}) = \left\{
\begin{aligned}
&\ell_k, & &i_k \geq 2; \\
&0, & &i_k = 1;
\end{aligned}\right.
\quad
r_k(\mathbf{i}) = \left\{
\begin{aligned}
&r_k, & &i_k \leq m_k - 1; \\
&0, & &i_k = m_k;
\end{aligned}\right.
$$
for $1 \leq k \leq n$, and $\mathbf{i} = (i_1, i_2, \ldots, i_n)$.

First we consider the case where the parameter $c = 0$. Let $\Theta: \mathbb{Z}_p^{m_1 \times m_2 \times \cdots \times m_n} \to \mathbb{Z}_p^{m_1 m_2 \cdots m_n}$ denote the transformation that designates the state $X = (X_{\mathbf{i}})_{1 \leq i_k \leq m_k, 1 \leq k \leq n}$ as a column vector with respect to the anti-lexicographic order. Set $T_1 = S_{m_1}(\ell_1, r_1)$ and
$$
T_k = I_{m_k} \otimes T_{k-1} + S_{m_k} (\ell_k, r_k) \otimes I_{\dim T_{k-1}} \quad \text{for} \quad 2 \leq k \leq n,
$$
where $\dim A$ refers to the dimension of the square matrix $A$. The following theorem is derived immediately.

\begin{theorem} \label{thm:n-dim-matrix-representation-commute-diagram}
The linear cellular automaton $\Phi_N$ over $\mathbb{Z}_p$ under null boundary condition is completely characterized by the matrix $T_n$. More explicitly, the diagram
$$
\xymatrix{
\mathbb{Z}_p^{m_1 \times m_2 \times \cdots \times m_n} \ar[rr]^{\Phi_N} \ar[d]_{\Theta} && \mathbb{Z}_p^{m_1 \times m_2 \times \cdots \times m_n} \ar[d]^{\Theta} \\
\mathbb{Z}_p^{m_1 m_2 \cdots m_n} \ar[rr]_{\mathbf{T}} && \mathbb{Z}_p^{m_1 m_2 \cdots m_n}}
$$
commutes, where $\mathbf{T} y = T_n y \pmod{p}$ for every $y \in \mathbb{Z}_p^{m_1 m_2 \cdots m_n}$.
\end{theorem}

Since $\Theta$ is a one-to-one correspondence, the following statements are equivalent.
\begin{enumerate}
\item $\Phi_N$ is reversible;
\item $T_n$ is invertible over $\mathbb{Z}_p$;
\item $0$ is not an eigenvalue of $T_n$ over $\mathbb{Z}_p$.
\end{enumerate}

\begin{theorem} \label{thm:n-dim-eigenvalue-set-decomposition}
Let $\mathbf{R}_k$ denote the collection of roots of $g_{m_k; \ell_k, r_k}^{[p]}(x)$ in the splitting field $\mathbb{E}$ for \{$g_{m_k; \ell_k, r_k}^{[p]}\}_{1 \leq k \leq n}$, where $g$ is defined in \eqref{eq:charpoly-for-Sj-t1-t2}. Then the set $\mathbf{E}_{T_n}$ of eigenvalues of $T_n$ is
$$
\mathbf{E}_{T_n} = \mathbf{R}_1 + \mathbf{R}_2 + \cdots + \mathbf{R}_n,
$$
where ``$+$'' refers to the Minkowski sum.
\end{theorem}

Similar to Theorems \ref{thm:main-theorem-algorithm-for-JT-and-Inverse} and \ref{thm:general-main-theorem-algorithm-for-JT-and-Inverse}, there is an algorithm for the computation of the generalized Jordan form of $T_n$ and $T_n^{-1}$, if it exists. Rather than describing the steps of the corresponding algorithm, which is analogous to the discussion in the proof of Theorems \ref{thm:main-theorem-algorithm-for-JT-and-Inverse} and \ref{thm:general-main-theorem-algorithm-for-JT-and-Inverse}, the following theorem illustrates the formula of the desired matrix that derives the generalized Jordan form of $T_n$.

\begin{theorem} \label{thm:n-dim-U-Jordan-form-nearest-nbd}
Let $U_{m_k}(\ell_k, r_k) \in \mathcal{M}_{m_k}(\mathbb{E})$ be the matrix that transforms $S_{m_k}(\ell_k, r_k)$ to its canonical Jordan form over the splitting field $\mathbb{E}$ for $\{g_{m_k; \ell_k, r_k}^{[p]}\}_{1 \leq k \leq n}$. Define
$$
U_{T_n} = U_{m_n}(\ell_n, r_n) \otimes U_{m_{n-1}}(\ell_{n-1}, r_{n-1}) \otimes \cdots \otimes U_{m_1}(\ell_1, r_1);
$$
then $U_{T_n}$ is invertible and $U_{T_n}^{-1} T_n U_{T_n}$ is a generalized Jordan form. Furthermore, $U_{T_n}^{-1} T_n U_{T_n}$ is a canonical Jordan form if and only if $S_{m_k}(\ell_k, r_k)$ is diagonalizable over $\mathbb{E}$ for $k \geq 2$, and $T_n$ is diagonalizable over $\mathbb{E}$ if and only if $S_{m_k}(\ell_k, r_k)$ is diagonalizable over $\mathbb{E}$ for all $k$.
\end{theorem}

\begin{remark}
In the case where $c \neq 0$, we substitute $T_1$ as $T_1' = T_1 + c I_{\dim T_1}$, then Theorems \ref{thm:n-dim-matrix-representation-commute-diagram} and \ref{thm:n-dim-U-Jordan-form-nearest-nbd} still work; notably, $U_{m_1}(\ell_1, r_1)$ should also be substituted as $U_{m_1}'$ which transforms $T_1'$ to its canonical Jordan form. Furthermore, Theorem \ref{thm:n-dim-eigenvalue-set-decomposition} remains to be true after replacing $\mathbf{R}_1$ by $\mathbf{R}_1'$, the collection of roots of $g_{m_1; \ell_1, r_1}^{[p]}(x-c)$.
\end{remark}

\subsection{Prolonged $\eta$-Nearest Neighborhood} \label{subsec:multidimension-eta-nesrest-nbd}

This subsection extends the previous discussion to $n$-dimensional linear cellular automata with the prolonged $\eta$-nearest neighborhood for $\eta \geq 2$. Set
$$
\mathcal{N} = \{v \in \mathbb{Z}^n: v = \lambda e_k, \text{ where } 1 \leq k \leq n, -\eta \leq \lambda \leq \eta\}.
$$
Fix $c, \ell_{k,j}, r_{k,j} \in \mathbb{Z}_p$ for $1 \leq k \leq n$ and $1 \leq j \leq \eta$; define $\phi: \mathbb{Z}_p^{\mathcal{N}} \to \mathbb{Z}_p$ as
$$
\phi(y_{\mathcal{N}}) = c y_{\mathbf{0}} + \sum_{k=1}^n \sum_{\lambda=1}^{\eta} (\ell_{k,j} y_{- \lambda e_k} + r_{k,j} y_{\lambda e_k}) \pmod{p}
$$
An $n$-dimensional linear cellular automaton $\Phi: \mathbb{Z}_p^{\mathbb{Z}^n} \to \mathbb{Z}_p^{\mathbb{Z}^n}$ with the prolonged $\eta$-nearest neighborhood is defined as
$$
\Phi(X)_{\mathbf{i}} = \phi(X_{\mathbf{i} + \mathcal{N}}) = c X_{\mathbf{i}} + \sum_{k=1}^n \sum_{\lambda=1}^{\eta} (\ell_{k,j} X_{\mathbf{i} - \lambda e_k} + r_{k,j} X_{\mathbf{i} + \lambda e_k}) \pmod{p}
$$
for every $\mathbf{i} \in \mathbb{Z}^n$. Given $m_1, m_2, \ldots, m_n \in \mathbb{N}$, $m_k \geq 2$ for $1 \leq k \leq n$, a linear cellular automaton under null boundary condition is described as
$$
\Phi_N(X)_{\mathbf{i}} = c X_{\mathbf{i}} + \sum_{k=1}^n \sum_{\lambda=1}^{\eta} (\ell_{k,j}(\mathbf{i}) X_{\mathbf{i} - \lambda e_k} + r_{k,j}(\mathbf{i}) X_{\mathbf{i} + \lambda e_k}) \pmod{p}
$$
where
$$
\ell_{k,j}(\mathbf{i}) = \left\{
\begin{aligned}
&\ell_{k,j}, & &i_k \geq \eta+1; \\
&0, & &\hbox{otherwise;}
\end{aligned}\right.
\quad
r_{k,j}(\mathbf{i}) = \left\{
\begin{aligned}
&r_{k,j}, & &i_k \leq m_k - \eta; \\
&0, & &\hbox{otherwise;}
\end{aligned}\right.
$$
for $1 \leq k \leq n, 1 \leq j \leq \eta$, and $\mathbf{i} = (i_1, i_2, \ldots, i_n)$.

For $i, j \in \mathbb{N}$ with $j < i$, define
$$
S_{i}(\alpha_j, \ldots, \alpha_1, \beta_1, \ldots, \beta_j) = \begin{pmatrix}
0 & \beta_1 & \beta_2 & \cdots & \beta_j & 0 & \cdots & 0 \\ 
\alpha_1 & 0 & \beta_1 & \beta_2 & \cdots & \beta_j & \cdots & 0 \\ 
\alpha_2 & \alpha_1 & 0 & \beta_1 & \ddots & \ddots & \ddots & \vdots \\ 
\vdots & \alpha_2 & \alpha_1 & 0 & \ddots & \ddots & \ddots & \beta_j \\ 
\alpha_j & \ddots & \ddots & \ddots & \ddots & \ddots & \ddots & \vdots \\ 
0 & \ddots & \ddots & \ddots & \ddots & 0 & \beta_1 & \beta_2 \\ 
\vdots & \ddots & \ddots & \ddots & \ddots & \alpha_1 & 0 & \beta_1 \\ 
0 & \cdots & 0 & \alpha_j & \cdots & \alpha_2 & \alpha_1 & 0
\end{pmatrix}_{i \times i};
$$
notably, $S_{i}(\alpha_j, \ldots, \alpha_1, \beta_1, \ldots, \beta_j)$ is a Toeplitz matrix (\cite{Gray-2006,GC-2012}).  Set
$$
T_1 = S_{m_1}(\ell_{1,\eta}, \ldots, r_{1,\eta}) + c I_{m_1}
$$
and
$$
T_k = I_{m_k} \otimes T_{k-1} + S_{m_k} (\ell_{k,\eta}, \ldots, r_{k,\eta}) \otimes I_{\dim T_{k-1}} \quad \text{for} \quad 2 \leq k \leq n;
$$
it is seen immediately that $T_n$ is the matrix representation of $\Phi_N$. In other words, Theorem \ref{thm:n-dim-matrix-representation-commute-diagram} is extended to the $\eta$-nearest neighborhood case. The extension of Theorems \ref{thm:n-dim-eigenvalue-set-decomposition} and \ref{thm:n-dim-U-Jordan-form-nearest-nbd} can be established analogously; we skip the description for the compactness of this investigation.

\section{Conclusion and Future Work} \label{sec:conclusion}

In this paper, we investigate the reversibility problem of multidimensional linear cellular automata under null boundary conditions. It follows that the matrix representation of $n$-dimensional linear cellular automata with the prolonged $\eta$-nearest neighborhood is the Kronecker sum of $n$ smaller matrices, each of which is a Toeplitz matrix. Such a cellular automaton is reversible if and only if the Minkowski sum of the sets of eigenvalues of block Toeplitz matrices contains no zero. When the cellular automaton is reversible, we provide an algorithm for deriving its reverse rule.

The proposed method significantly reduces the computational cost when the number of cells is large or when the dimension $n$ is large. Furthermore, the dynamical behavior of a multidimensional linear cellular automaton under null boundary condition is revealed by elucidating the properties of block Toeplitz matrices.

We remark that the elucidation in this work can extend to the investigation of cellular automata under periodic boundary conditions with a minor modification. The discussion is analogous, hence it is omitted. Furthermore, Dennunzio \emph{et al.} \cite{DFW-TCS2014} characterize the properties, such as quasi-expansivity and closing property, of multidimensional cellular automata by transposing them into some specific one-dimensional systems. It is of interest how the results obtained in the present paper are reflected on its associated one-dimensional cellular automaton. The related work is under preparation.

\section*{Acknowledgment}
We would like to express our deep gratitude for the anonymous referees' valuable and constructive comments, which have significantly improved the quality and readability of this paper.

\bibliographystyle{amsplain}
\bibliography{../../grece.bib}

\providecommand{\bysame}{\leavevmode\hbox to3em{\hrulefill}\thinspace}
\providecommand{\MR}{\relax\ifhmode\unskip\space\fi MR }
\providecommand{\MRhref}[2]{%
  \href{http://www.ams.org/mathscinet-getitem?mr=#1}{#2}
}
\providecommand{\href}[2]{#2}
\begin{thebibliography}{10}

\bibitem{ALI+-CNSNS2013}
A.~A. Abdoa, S.~Lianb, I.~A. Ismailc, M.~Amina, and H.~Diaba, \emph{A
  cryptosystem based on elementary cellular automata}, Commun. Nonlinear Sci.
  Numer. Simul. \textbf{18} (2013), 136--147.

\bibitem{AP-JCSS1972}
S.~Amoroso and Y.~N. Patt, \emph{Decision procedures for surjectivity and
  injectivity of parallelmaps for tessellation structures}, J. Comput. System
  Sci. \textbf{6} (1972), 448--464.

\bibitem{Bennett-IJRD1973}
C.~H. Bennett, \emph{Logical reversibility of computation}, IBM J Res. Develop.
  \textbf{17} (1973), 525--532.

\bibitem{CMC+-ITIP2011}
L.~Cappellari, S.~Milani, C.~{Cruz-Reyes}, and G.~Calvagno, \emph{Resolution
  scalable image coding with reversible cellular automata}, IEEE Trans. Image
  Process. \textbf{20} (2011), 1461--1468.

\bibitem{CC-IS2016}
C.-H. Chang and H.~Chang, \emph{On the {Bernoulli} automorphism of reversible
  linear cellular automata}, Inform. Sci. \textbf{345} (2016), 217--225.

\bibitem{CS-2016}
C.-H. Chang and J.-Y. Su, \emph{Reversibility of linear cellular automata on
  cayley trees with periodic boundary condition}, arXiv:1603.01679, 2016.

\bibitem{CAS-JSP2011}
Z.~Cinkir, H.~Ak{\i}n, and I.~Siap, \emph{Reversibility of {1D} cellular
  automata with periodic boundary over finite fields $\mathbb{Z}_p$}, J. Stat.
  Phys. \textbf{143} (2011), 807--823.

\bibitem{RS-AMC2011}
A.~M. del Rey and G.~R. S\'{a}nchez, \emph{Reversibility of linear cellular
  automata}, Appl. Math. Comput. \textbf{217} (2011), 8360–8366.

\bibitem{DFW-TCS2014}
A.~Dennunzio, E.~Formenti, and M.~Weiss, \emph{Multidimensional cellular
  automata: closing property, quasi-expansivity, and (un)decidability issues},
  Theor. Comput. Sci. \textbf{516} (2014), 40--59.

\bibitem{ER-AMC2007a}
L.~H. Encinas and A.~M. del Rey, \emph{Inverse rules of {ECA} with rule number
  150}, Appl. Math. Comput. \textbf{189} (2007), 1782--1786.

\bibitem{Gray-2006}
R.~M. Gray, \emph{Toeplitz and circulant matrices: A review}, Foundations and
  Trends in Communications and Information, Now Publishers Inc, 2006.

\bibitem{GC-2012}
J.~Gutierrez-Gutierrez and P.~M. Crespo, \emph{Block toeplitz matrices:
  Asymptotic results and applications}, Foundations and Trends in
  Communications and Information, Now Publishers Inc, 2012.

\bibitem{HT-ITNN2011}
T.~Hishiki and H.~Torikai, \emph{A novel rotate-and-fire digital spiking neuron
  and its neuron-like bifurcations and responses}, IEEE Trans. Neural Netw.
  \textbf{22} (2011), 752--767.

\bibitem{HJ-1994}
R.~A. Horn and C.~R. Johnson, \emph{Topics in matrix analysis}, Cambridge
  University Press, 1994.

\bibitem{ION-JCSS1983}
M.~Ito, N.~Osato, and M.~Nasu, \emph{Linear cellular automata over
  $\mathbb{Z}_m$}, J. Comput. System Sci. \textbf{27} (1983), 125--140.

\bibitem{Kari-PD1990}
J.~Kari, \emph{Reversibility of {2D} cellular automata is undecidable}, Physica
  D \textbf{45} (1990), 386--395.

\bibitem{Kari-JCSS1994}
\bysame, \emph{Reversibility and surjectivity problems of cellular automata},
  J. Comput. System Sci. \textbf{48} (1994), 149--182.

\bibitem{Kari-TCS2005}
\bysame, \emph{Theory of cellular automata: {A} survey}, Theoret. Comput. Sci.
  \textbf{334} (2005), 3--33.

\bibitem{KHP-IATCBB2012}
N.~Kazmi, M.~A. Hossain, and R.~M. Phillips, \emph{A hybrid cellular automaton
  model of solid tumor growth and bioreductive drug transport}, IEEE ACM Trans.
  Comput. Biol. Bioinfo. \textbf{9} (2012), 1595--1606.

\bibitem{KBT+-JID2013}
S.~Kippenberger, A.~Bernd, D.~Tha\c{c}i, R.~Kaufmann, and M.~Meissner,
  \emph{Modeling pattern formation in skin diseases by a cellular automaton},
  J. Invest. Dermatol. \textbf{133} (2013), 567--571.

\bibitem{KSA-TJM2016}
M.~E. K\"{o}ro\u{g}lu, I.~Siap, and H.~Ak{\i}n, \emph{The reversibility problem
  for a family of two-dimensional cellular automata}, Turkish J. Math.
  \textbf{40} (2016), 665--678.

\bibitem{Landauer-IJRD1961}
R.~Landauer, \emph{Irreversibility and heat generation in the computing
  process}, IBM J Res. Develop. \textbf{5} (1961), 183--191.

\bibitem{MM-JCSS1998}
G.~Manzini and L.~Margara, \emph{Invertible linear cellular automata over
  $\mathbb{Z}_m$: {Algorithmic} and dynamical aspects}, J. Comput. System Sci.
  \textbf{56} (1998), 60--97.

\bibitem{Morit-2012}
K.~Morita, \emph{Reversible cellular automata}, Handbook of Natural Computing,
  Springer-Verlag Berlin Heidelberg, 2012, pp.~231--257.

\bibitem{MH-IT1989}
K.~Morita and M.~Harao, \emph{Computation universality of 1 dimensional
  reversible (injective) cellular automata}, IEICE Trans. \textbf{E72} (1989),
  758--762.

\bibitem{MM-2007}
G.~L. Mullen and C.~Mummert, \emph{Finite fields and applications}, American
  Mathematical Society, 2007.

\bibitem{Nasu-TAMS2002}
M.~Nasu, \emph{The dynamics of expansive invertible onesided cellular
  automata}, Trans. Am. Math. Soc. \textbf{354} (2002), 4067--4084.

\bibitem{NY-JPAG2004}
A.~Nobe and F.~Yura, \emph{On reversibility of cellular automata with periodic
  boundary conditions}, J. Phys. A-Math. General \textbf{37} (2004),
  5789--5804.

\bibitem{Orteg-1987}
J.~M. Ortega, \emph{Matrix theory}, Plenum Press, 1987.

\bibitem{Rom-2006}
S.~Roman, \emph{Field theory}, 2 ed., Springer, 2006.

\bibitem{SUA+-AMM2015}
U.~Sahin, S.~Ugus, H.~Ak{\i}n, and I.~Siap, \emph{Three-state von {Neumann}
  cellular automata and pattern generation}, Appl. Math. Model. \textbf{39}
  (2015), 2003--2024.

\bibitem{TMA+-IS2012}
J.~C. {Seck-Tuoh-Mora}, G.~J. Mart\'{i}nez, R.~Alonso-Sanz, and
  N.~Hern\'{a}ndez-Romero, \emph{Invertible behavior in elementary cellular
  automata with memory}, Inform. Sci. \textbf{199} (2012), 125 -- 132.

\bibitem{Toffoli-JCSS1977}
T.~Toffoli, \emph{Computation and construction universality of reversible
  cellular automata}, J. Comput. Syst. Sci. \textbf{15} (1977), 213--231.

\bibitem{VBT-IJSAC2011}
W.~Viriyasitavat, F.~Bai, and O.~K. Tonguz, \emph{Dynamics of network
  connectivity in urban vehicular networks}, IEEE J. Sel. Area Commun.
  \textbf{29} (2011), 515--533.

\bibitem{Yamagishi-FFA2013}
M.~Yamagishi, \emph{Elliptic curves over finite fields and reversibility of
  additive cellular automata on square grids}, Finite Fields Appl. \textbf{19}
  (2013), 105 -- 119.

\bibitem{YWX-IS2015}
B.~Yang, C.~Wang, and A.~Xiang, \emph{Reversibility of general 1d linear
  cellular automata over the binary field $\mathbb{Z}_2$ under null boundary
  conditions}, Inform. Sci. \textbf{324} (2015), 23--31.

\end{thebibliography}

\end{document}